\documentclass[11pt]{article}
\usepackage{latexsym,amssymb,enumerate,amsmath,epsfig,amsthm,dsfont,bm}
\usepackage{setspace,color}
\usepackage{tikz}
\usepackage{commath}
\usetikzlibrary{intersections,positioning,calc}
\usepackage{floatrow}
\usepackage{cite}
\usepackage{graphicx,subfigure}
\usepackage[ruled]{algorithm2e}
\usepackage{epstopdf}
\usepackage{tcolorbox}
\usepackage[multidot]{grffile}
\usepackage{comment}
\usepackage{multirow}
\usepackage[utf8]{inputenc}
\usepackage[nohead, nomarginpar, margin=1in, foot=.25in]{geometry}
\usepackage{bbm}
\usepackage{tikz}
\usepackage{mathtools}
\usepackage{enumitem}

\theoremstyle{definition}

\newtheorem{theorem}{Theorem}[section]
\newtheorem{corollary}{Corollary}[theorem]
\newtheorem{lemma}{Lemma}
\newtheorem{proposition}{Proposition}

\DeclareMathOperator{\E}{\mathbb{E}}

\DeclareMathOperator*{\argmin}{arg\,min}

\newcommand{\btau}{\bm{\tau}}
\newcommand{\bbeta}{\bm{\beta}}
\newcommand{\bF}{\mathbf{F}}
\newcommand{\mS}{\mathcal{S}}
\newcommand{\bu}{\mathbf{u}}

\newcommand{\rn}[1]{%
  \textup{\lowercase\expandafter{\romannumeral#1}}%
}

\DeclareFontEncoding{FMS}{}{}
\DeclareFontSubstitution{FMS}{futm}{m}{n}
\DeclareFontEncoding{FMX}{}{}
\DeclareFontSubstitution{FMX}{futm}{m}{n}
\DeclareSymbolFont{fouriersymbols}{FMS}{futm}{m}{n}
\DeclareSymbolFont{fourierlargesymbols}{FMX}{futm}{m}{n}
\DeclareMathDelimiter{\VERT}{\mathord}{fouriersymbols}{152}{fourierlargesymbols}{147}

\theoremstyle{definition}

\newtheorem{remark}{Remark}[section]

\date{} 

\begin{document}

\title{Asymptotic Theory of $\ell_1$-Regularized PDE Identification from a Single Noisy Trajectory} 
\author{Yuchen He \footnote{Yuchen He and Namjoon Suh contributed equally to this article, listed in alphabetical order.} \thanks{School of Mathematics, Georgia Institute of Technology, 686 Cherry St NW, Atlanta, GA, USA.},
Namjoon Suh \footnotemark[1]  \thanks{School of Industrial and Systems Engineering, Georgia Institute of Technology, 755 Ferst Dr, Atlanta, GA, USA.}, 
Xiaoming Huo \footnotemark[3], 
Sung Ha Kang \footnotemark[2], 
Yajun Mei\footnotemark[3]}

\maketitle
\begin{abstract}
We prove the support recovery for a general class of linear and nonlinear evolutionary partial differential equation (PDE) identification from a single noisy trajectory using $\ell_1$ regularized Pseudo-Least Squares model~($\ell_1$-PsLS).
In any associative $\mathbb{R}$-algebra generated by finitely many differentiation operators that contain the unknown PDE operator, applying $\ell_1$-PsLS to a given data set yields a family of candidate models with coefficients $\mathbf{c}(\lambda)$ parameterized by the regularization weight $\lambda\geq 0$.
The trace of $\{\mathbf{c}(\lambda)\}_{\lambda\geq 0}$ suffers from high variance due to data noises and finite difference approximation errors.
We provide a set of sufficient conditions which guarantee  that, from a single trajectory data denoised by a Local-Polynomial filter, the support of $\mathbf{c}(\lambda)$ asymptotically converges to the true signed-support associated with the underlying PDE for sufficiently many data and a certain range of $\lambda$. We also show various numerical experiments to validate our theory.
\end{abstract}

\section{Introduction}
Differential equations are widely used to describe many interesting phenomena arising in scientific fields, including physics~\cite{schrodinger1926undulatory}, social sciences~\cite{black1973pricing}, biomedical sciences~\cite{haskovec2018ode}, and economics~\cite{achdou2014partial}, just to name a few. 
The forward problem of solving equations or simulating state variables for differential models has been extensively studied either theoretically or numerically in literature.
We consider an inverse problem of learning a Partial Differential Equations (PDE) model.

More specifically, we assume that the governing PDE is a multi-variate polynomial of a subset of a prescribed dictionary containing different differential terms.
Let $u(x,t):\mathbb{R}\times[0,+\infty)\to\mathbb{R}$ be a real-valued function, where $x$ be the spatial and $t$ be the temporal variables.
Suppose that within a bounded region of $\mathbb{R}\times[0,+\infty)$, $u(x,t)$ satisfies an evolutionary PDE: 
\begin{align}\label{eq.PDE0}
    \partial_{t} u = &\mathcal{F}(u,\partial_xu,\partial_{x}^{2}u,\dots,),\quad
    \forall (x,t)\in \Omega\subseteq \mathbb{R}\times[0,+\infty). \;
\end{align}
Here, $\partial_{t} u$ (or $u_{t}$) denotes the partial derivative of $u$ with respect to temporal variable, $t$; for $p=0,1,2,\dots,\partial_{x}^{p}u$ denotes the $p$-th order partial derivative of $u$ with respect to spatial  variable, $x$; 
$\mathcal{F}$ is an unknown polynomial mapping, and $\Omega$ is a bounded open subset of space-time domain.
This format encloses various important classes of PDEs, 
e.g., advection-diffusion-decay equation characterizing pollutant distribution in fluid,
Burgers’ equation modeling the traffic flow \cite{musha1978traffic}, Kolmogorov-Petrovsky-Piskunov equation describing phase transitions \cite{tikhomirov1991study}, and Korteweg-de-Vries equation simulating the shallow water dynamics \cite{newell1985solitons}. 

In our work, $\mathcal{F}$ is assumed to be a linear map, parametrized by a sparse vector $\bbeta^*\in\mathbb{R}^{K}$: that is,
$u_{t}$ is represented as a linear combination of the arguments of $\mathcal{F}$, and only a few from a large set of potential functions are relevant with $u_{t}$.
Our goal is to estimate the correct non-zero indices of $\bbeta^*$, given a single noisy trajectory of the function $u(x,t)$.
Readers can refer to Subsection~\ref{setting} for more detailed descriptions on the structural assumptions on $\mathcal{F}$, $\bbeta^*$, and noisy trajectory.

This problem setting naturally leads us to develop a two-stage method for the PDE identification based on Local-Polynomial smoothing and the $\ell_{1}$-regularized Pseudo Least Square ($\ell_{1}$-PsLS) method.
In the first stage, from a given noisy observation, we propose to estimate the underlying bi-variate function $u(x,t)$ and its partial derivatives with respect to its spatial and temporal dimensions via the Local-Polynomial fitting \cite{fan1997local,fan2018local}. 
In the second stage, with the constructed functions through Local-Polynomial regression, we propose to identify the correct differential terms and estimate model parameters via an  $\ell_{1}$-regularized Pseudo Least-Square method.

We note that the two-stage method with Local-Polynomial regression has been applied in the Ordinary Differential Equations (ODE) setting. 
Specifically, the paper~\cite{liang2008parameter} established the consistency and asymptotic normality of the pseudo least square estimator in the ODE setting, where they used Local-Polynomial regression to estimate the state variables from the noisy data. 
Similarly,~\cite{chen2008efficient,chen2008estimation} studied the parameter estimation of ODE models with varying coefficients. 
However, these literature focused on estimating model parameters, rather than on selecting correct  differential models.
In the context of PDE,~\cite{bar1999fitting} studied PDE identification problems, using two-stage method.
Authors of the paper modeled unknown PDEs using multivariate polynomials of sufficiently high order, and the best fit was chosen by minimizing the least squares error of the polynomial approximation.
Nonetheless, $\ell_{1}$ penalization for model selection was not used, and theoretical justification for their method remains underdeveloped.

From the theoretical point of view, our paper is the first work to propose the 
method, $\ell_{1}$-PsLS, with a provable guarantee in the PDE recovery problem.
Our main theoretical contribution is to establish sufficient conditions for \textbf{\textit{signed-support recovery}} of
the proposed $\ell_{1}$-PsLS in PDE identification problems.
It is worth noting that the signed-support recovery is a slightly stronger criterion than the support recovery, 
where its primary goal is not limited to finding the non-zero indices of $\bbeta^*$, but also aims at recovering the correct signs of the selected coefficients.
Ensuring the correct signed-support recovery of governing dynamical system has an important practical implication since many PDEs are sensitive to the signs of coefficients. 
For example, changing the sign of the advection term in the transport equation reverses the moving direction, and inverting the sign of the Laplacian term of heat equation leads to instability of the system of interest.

Our theorem states that following two main conditions are sufficient for the signed-support recovery of $\ell_{1}$-PsLS: 
(\rn{1}) \textbf{\textit{mutual incoherence condition}} among the arguments of the map $\mathcal{F}$, and 
(\rn{2}) \textbf{\textit{$\bbeta^*_{\text{min}}$-condition}} on $\bbeta^*$. 
The first condition states that a large number of irrelevant predictors cannot exhibit an overly strong influence on the subset of relevant predictors.
The second condition says that the minimum absolute value of non-zero entries of $\bbeta^*$ should be greater than a certain threshold. 
These conditions appear in the statistical literature on the signed-support/support recovery of Lasso~\cite{tibshirani1996regression,wainwright2009sharp,jia2013lasso,buhlmann2011statistics} in linear regression problems, and our work rigorously shows that these are also essential for the signed-support recovery of PDE identification problems.

We employ Primal-Dual Witness (PDW) construction~\cite{wainwright2009sharp} as the main proof technique for the theorem.
PDW construction is a popular mathematical technique for certifying variable-selection consistency of $\ell_{1}$-penalized M-estimation problems including Lasso.
See~\cite{ravikumar2010high,ravikumar2009sparse,ravikumar2011high,obozinski2008union,wang2013block,jalali2010dirty}.
For reader's convenience, we provide a brief introduction of the technique in the Appendix~\ref{app.pdw}.
However, we want to emphasize that our Theorem is not a direct result of the trivial application of the PDW construction.
Our problem settings are different from those of the work~\cite{wainwright2009sharp} in two aspects, which add some delicacies to our proof: 
\begin{itemize}
    \item As will be detailed in Subsection \ref{method_intro}, the distribution of residual vector $\btau$ is unknown, and neither mean $0$ nor independent in our setting.
    On the contrary, in the work of~\cite{wainwright2009sharp}, each entry of the residual vector is assumed to follow centered Gaussian with $\sigma^2>0$ variance and independent with the others.
    
    \item In the $\ell_{1}$-PsLS method, the feature matrix obtained via Local-Polynomial fitting from noisy data is always random and has dependent rows uniquely determined through the underlying PDE.
    On the other hand, \cite{wainwright2009sharp} divided their analysis into two cases, where the feature matrix $X$ is either deterministic or random.
    When $X$ is random, it is assumed to be a Gaussian ensemble with independent rows, whose covariance matrix satisfies mutual incoherence condition.
\end{itemize}
\noindent \textbf{Organization.}
The remainder of the paper is organized as follows.
Some related literature with our work are reviewed and discussed in Section~\ref{sec_related_work}.
In Section~\ref{Sec3}, we formally define our problem by imposing some specific structural assumptions on $\mathcal{F}$ and propose a $\ell_{1}$-PsLS method for PDE identification.
In Section~\ref{Sec4}, the main theorem of our work is given on the signed-support recovery of $\ell_{1}$-PsLS with the mutual incoherence assumption on the feature matrix $\bF$, and we provide a high-level outline of the proof.
Section~\ref{Sec5} is devoted to provide a similar result with that of the one in the main theorem in Section~\ref{Sec4} under milder assumption: that is, mutual incoherence assumption is imposed on the estimated feature matrix $\widehat{\bF}$; 
an overview of proof is furnished. 
Related technical difficulties for the proof and main technical contribution of the paper are also given.
Section~\ref{Sec6} provides two Lemmas for completing the proof of the main Theorem by linking the mutual incoherence assumption with ground-truth $\bF$ to its sampled version. 
In Section~\ref{sec_num}, we show various numerical examples to validate and demonstrate different aspects of our method. 
We conclude this paper in Section~\ref{sec_conclusion} with some discussion. 
\\ \\
\noindent \textbf{Notation.} 
We use the following notation for asymptotics: For sufficiently large $n$, we write $f(n)=\mathcal{O}(g(n))$, if there exists a constant $K>0$ such that $f(n) \leq Kg(n)$,
and $f(n)=\Omega(g(n))$ if $f(n)\geq K'g(n)$ for some constant $K'>0$. 
The notation $f(n)=\Theta(g(n))$ means that $f(n)=\mathcal{O}(g(n))$ and $f(n)=\Omega(g(n))$.
We adopt bold lower-case letters for vectors and bold upper-case letters for matrices. 
For a vector $\mathbf{v}\in\mathbb{R}^n$,  $\|\mathbf{v}\|_1:={\sum_{i=1}^{n}}|v_i|$, $\|\mathbf{v}\|_2:=\sqrt{\sum_{i=1}^{n}v_i^2}$, and $\|\mathbf{v}\|_\infty:=\underset{1\leq i \leq n}{\max}|v_i|$. For a matrix $\mathbf{A}\in\mathbb{R}^{n\times m}$, $\mathbf{A}^T$ denotes its transpose, $\VERT\mathbf{A}\VERT_2:=\max_{\forall \|x\|_{2}=1}\|Ax\|_{2}$, $\VERT\mathbf{A}\VERT_\infty:=\underset{1\leq i \leq n}{\max}\sum_{j=1}^{m}|A_{i,j}|$,
$\VERT\mathbf{A}\VERT_{\infty,\infty}:=\underset{1\leq i,j \leq n}{\max}|A_{i,j}|$, and $\VERT\mathbf{A}\VERT_{F}:=\sqrt{\sum_{i=1}^{n}\sum_{j=1}^{m}A_{i,j}^{2}}$.

\section{Related Works}\label{sec_related_work}
Our work is relevant to various topics in the fields of applied mathematics and statistics. 
Among them, we provide two most closely related topics: (\rn{1}) Regression-based framework for PDE identification, and (\rn{2}) Some theoretical results of support-recovery of Lasso~\cite{tibshirani1996regression} in linear regression setting.
In this Section, we denote $K$ as the problem dimension, $s$ as the number of non-zero entries of model parameter, and $n$ as the number of observations. 
\\ \\
\noindent\textbf{Regression-based Methods.}
Recently, various regression-based frameworks have been developed and applied for model selection and parameter estimation of dynamic data.
A sparsity-promoting method was proposed in~\cite{brunton2016discovering} for extracting the governing dynamical system, by comparing the computed velocity to a large set of potential trial functions. 
Under the over-determined systems of linear equations (i.e.,$n \gg k$), the authors developed a sequential-thresholded least-square method to select the correct nonlinear functions. 
In the follow-up study,~\cite{cortiella2021sparse} devised a weighted-$\ell_{1}$-regularized least squares solver for improving the accuracy and robustness of the approach introduced by~\cite{brunton2016discovering} in the presence of state-measurement noise. 
Several papers~\cite{kang2019ident,schaeffer2017learning,rudy2017data} also suggested sparse regression frameworks for PDE identification problems over spatial-temporal data. 
Specifically,~\cite{schaeffer2017learning} studied the model selection problem via Lasso under the PDE context. The author empirically showed that the method works well in various important equations such as Burgers' equation, Navier-Stokes equation, Swift-Hohenberg equation.
Recently,~\cite{kang2019ident} considered PDE identification problem using numerical time evolution. The authors utilized Lasso to select candidate monomials, then proposed the time evolution error to select the underlying true model.
Unlike the previously mentioned literature, which was mostly empirical,~\cite{schaeffer2018extracting} provided a provable guarantee on the usage of $\ell_{1}$-norm for PDE identification problems, based upon the theoretical results from compressive sensing. 
Interestingly, this work imposed the \textit{incoherence property} on the feature matrix and employed the Legendre-transform on the columns of the matrix to ensure that the property holds for every PDE recovery problem of interest.
Our work imposes mutual incoherence assumption on the feature matrix, which is an analogous notion of the incoherence property. 
However, the important difference between our paper and~\cite{schaeffer2018extracting} is that our work only allows a single trajectory, whereas~\cite{schaeffer2018extracting}’s theorem requires $\Omega(s\log K)$ bursts of initialization for the exact recovery of the underlying PDE.
\\ \\
\noindent\textbf{Support Recovery in Statistics.}
Support recovery or variable selection problems of Lasso have a long and intensive history in the statistical literature. 
In the noiseless setting, many researchers~\cite{chen2001atomic,donoho2001uncertainty,donoho2005stable,feuer2003sparse,candes2005decoding,candes2006robust} established different sufficient conditions for either the deterministic or random predictors for the support recovery problems of linear systems via the $\ell_{1}$-norm.

Since our work falls into the category of noisy setting, we focus more on reviewing the body of work in the noisy setting.
In~\cite{knight2000asymptotics}, authors studied the asymptotic behavior of the Lasso-type estimator with fixed dimension $K$ under the general centered i.i.d. noises with variance $\sigma^2>0$.  
Both~\cite{tropp2006just} and~\cite{donoho2005stable} independently developed sufficient conditions for the support of Lasso estimator to be contained within true support of the sparse model. 
Under a more general setting, when the exterior noise is i.i.d.~with finite moments,~\cite{zhao2006model} showed that the Irrepresentable Condition~\cite{fuchs2005recovery} is almost necessary and sufficient for Lasso's signed-support recovery for fixed $K$ and $s$. 
Furthermore, under the Gaussian noise assumption, they showed that Lasso can still achieve signed-support recovery when $K$ is allowed to grow exponentially faster than $n$. 
In a non-asymptotic setting,~\cite{wainwright2009sharp} established the sharp relationship of $n$, $K$, and $s$, required for the exact sign consistency of Lasso, where $K$ and $s$ are allowed to grow as $n$ increases under mutual incoherence condition. 
Using a similar technique in~\cite{wainwright2009sharp}, the paper~\cite{jia2013lasso} studied Lasso under Poisson-like model with heteroscedastic noise and show that irrepresentable condition can serve as a necessary and sufficient condition for signed-support recovery in their setting.
In the context of graphical model,~\cite{meinshausen2006high,ravikumar2008model} analyzed the model selection consistency of Gaussian graphical models, and~\cite{ravikumar2010high} showed the signed-support recovery of Ising models. 
See~\cite{fan2010selective} for a more comprehensive overview on this topic.

\begin{remark}
Our work is of asymptotic nature with fixed $K$ and $s$, while the number of grid points of the observed trajectory tends to infinity in both the spatial and the temporal dimensions.
\end{remark}

\section{PDE Identification via $\ell_{1}$-PsLS} \label{Sec3}
In Subsection \ref{setting}, we provide concrete problem settings on the governing PDE of \eqref{eq.PDE0} and the observed trajectory. 
Then, specific settings of the Local-Polynomial regression for the estimations of state variables in our paper are provided in Subsection \ref{local}.
Lastly, we propose a two-stage $\ell_{1}$-regularized Pseudo Least Square method for PDE identification in Subsection \ref{method_intro}.

\subsection{Problem Setting and Notations} \label{setting}
Based on the general form \eqref{eq.PDE0}, we take $(x,t)\in [0,X_{\max}) \times [0,T_{\max})$ for some finite constants $0<X_{\max},T_{\max}<\infty$.
It is assumed that the underlying mapping $\mathcal{F}$ is a degree $2$ polynomial\footnote{It should be noted that our setting can be generalized to higher-degrees of polynomials and functions with multiple spatial dimensions.} parametrized by a coefficient vector $\bbeta^*=(\beta_{0}^*,\beta_{1}^*,\dots,\beta^*_{p,q},\dots)$ with real entries, that is,
\begin{align}
    & u_{t}(x,t)=\mathcal{F}(u,\partial_xu,\partial_{x}^{2}u,\dots,;\bbeta^*):=\beta_{0}^*+\beta_{1}^*u+\beta_{2}^*u_x+\beta_{3}^*u_{xx}+\dots
    +\beta^*_{p,q}\partial_x^{p}u \partial_x^{q}u+\dots .\label{eq.PDE1}
\end{align}
We call the monomials in the right-hand side of \eqref{eq.PDE1} as \textit{feature variables}.
We set a finite integer upper-bound, $P_{\max}>0$, for the possible orders of the partial derivatives of $u$ with respect to $x$ in \eqref{eq.PDE1}.
Hence, We assume that $\bbeta^*\in\mathbb{R}^{K}$, with $K=1+2(P_{\max}+1)+{{P_{\max}+1}\choose{2}}$;
consequently, constant and any term of the form $\partial_{x}^{p}u$ or $\partial_{x}^{p}u\partial_{x}^{q}u$, for $0\leq p, q \leq P_{\max}$, are contained in \eqref{eq.PDE1}.
Notice that many entries of $\bbeta^*$ can be zero.
We denote $\mathcal{S}(\bbeta^*):=\{0\leq j\leq K \mid \beta_{j}^*\neq 0\}$, or simply $\mathcal{S}$, as the support of the coefficient vector $\bbeta^*$, i.e., the set of indices of the non-zero entries.
Additionally, we denote $s$ as the cardinality of the set $\mathcal{S}$, i.e., $s:=|\mathcal{S}(\bbeta^*)|$.

The given data set $\mathcal{D}=\{\big(X_{i},t_{n},U_{i}^{n}\big)\mid ~i=0,\dots,M-1; n=0,\dots,N-1\} \subseteq \Omega\times\mathbb{R}$ consists of  
$M\times N$ data, where $M,N\in\mathbb{R}$, $M,N\geq 1$.
Each $(X_{i},t_{n})\in\Omega$ represents a space-time sampling grid point, and 
$U_{i}^{n}$ is a representation of $u(X_{i},t_{n})$ contaminated by additive Gaussian noise:
    \begin{align*}
        U_{i}^{n} = u(X_{i},t_{n}) + \nu_{i}^{n},\quad ~\nu_i^n\overset{\text{i.i.d.}}{\sim}\mathcal{N}(0,\sigma^2)\;,
    \end{align*}
whose second moment is uniformly bounded as follows: $\sup_{N,M\in\mathbb{R}}\max_{n,i}E\left|U_{i}^{n}\right|^{2}:=\eta^{2}<\infty$.
Here $\mathcal{N}(0,\sigma^2)$ denotes the centered normal distribution with variance $\sigma^{2}>0$.

\subsection{Local-Polynomial Regression Estimators for Derivatives} \label{local}
Given data $\{(X_i,t_n,U_i^n)\}$ with $i=0,1,\dots,M-1$ and  $n=0,1,\dots,N-1$, we employ a local quadratic regression to estimate $u_t(X_i,\cdot)$ for each fixed space point $X_i$ and use a Local-Polynomial with degree $p+1$ to estimate $\partial_x^pu(\cdot,t_n)$ at each temporal point $t_n$, for each degree $p=0,1,\dots,P_{\max}.$ More specifically, we solve the following  optimization problems:
\begin{align}
&\bigg\{\widehat{b}_j(X_i,t)\bigg\}_{j=0,1,2}=\argmin_{b_j(t)\in\mathbb{R},0\leq j \leq 2 }\sum_{n=0}^{N-1}\bigg(U_i^n-\sum_{j=0}^{2}b_j(t)(t_n-t)^j\bigg)^2\mathcal{K}_{h_N}\bigg(t_n-t\bigg)\;,\nonumber\\
&\quad\quad\text{for}~i=0,1,\dots,M-1\;;\label{eq_K_b}\\
&\bigg\{\widehat{c}^p_j(x,t_n)\bigg\}_{j=0,1,\dots,p+1}=\argmin_{c_j(t)\in\mathbb{R},0\leq j \leq p+1 }\sum_{i=0}^{M-1}\bigg(U_i^n-\sum_{j=0}^{p+1}c^p_j(t)(X_i-x)^j\bigg)^2\mathcal{K}_{w_{M}}\bigg(X_i-x\bigg)\;,\nonumber\\
&\quad\quad\text{for}~n=0,1,\dots,N-1\; \text{and}~p=0,1,\dots,P_{\max}.\label{eq_K_c}
\end{align}
and set $\widehat{u}_t(X_i,t) = \widehat{b}_1(X_i,t)$ and $\widehat{\partial_x^pu}(x,t_n ) = p!\widehat{c}^p_{p}(x,t_n)$. 
Here $h_N$ and $w_{p,M}$ denote the window width parameters, and $\mathcal{K}_w(z):=\mathcal{K}(z/w)/w$ for some kernel function $\mathcal{K}$ with window width $w>0$. 
Specific choices of the order of polynomial fit for the functions $\widehat{u}_{t}$ and $\widehat{\partial_{x}^{p}u}$ are to strike the balance between modeling bias and variance.
See Subsections 3.1 and 3.3 of Fan and Gijbels \cite{fan1997local} for more rigorous treatments on this topic.
Also the kernel $\mathcal{K}$ is assumed to be uniformly continuous and absolutely integrable with respect to Lebesgue measure on the real-line; $\mathcal{K}(z)\to 0$ as $|z|\to+\infty$; and  $\int |z\ln|z\|^{1/2}|dK(z)|<+\infty$.

Optimization problems \eqref{eq_K_b} and \eqref{eq_K_c} have closed-form solutions in the form of
weighted least square estimator. 
See Appendix~\ref{app.loc}.
However, for theoretical investigation, we employ the notion of \emph{equivalent kernel} \cite{fan1997local,fan2018local} to write the solutions as follows:
for any fixed spatial point $X_{i}$, $i=0,1,\dots,M-1$, $\widehat{u}_t(X_i,t)$ can be written as: 
\begin{align} \label{LP_t}
	\widehat{u_t}(X_i,t)&=\frac{1}{Nh^2_N}\sum_{n=0}^{N-1}\mathcal{K}_{2}^*\bigg(\frac{t_n-t}{h_N}\bigg)U_i^n \big\{ 1+o_{\mathbb{P}}(1)\big\}.
\end{align}
Similarly, for any fixed temporal point $t_n$, $n=0,1,\dots, N-1$, the estimation for the $p$-th order partial derivative takes the form: 
\begin{align} \label{LP_x}
    \widehat{\partial_x^pu}(x,t_n) = \frac{p!}{Mw_M^{p+1}}\sum_{i=1}^M\mathcal{K}^*_{p}\left(\frac{X_{i}-x}{w_M}\right)U_i^n\big\{ 1+o_{\mathbb{P}}(1)\big\}.
\end{align}
Here, $\mathcal{K}_{j}^*(z)=e_{j}^{T}S^{-1}(1,z,\dots,z^{p})^{T}K(z)$ is called an equivalent kernel, where $e_{j}$ denotes a unit vector with $1$ on the $j^{\text{th}}$ position; $S=(\int z^{l+s}\mathcal{K}(z)dz)_{0\leq l,s \leq p}$ is the moment matrix associated with kernel $\mathcal{K}$; and $o_{\mathbb{P}}(1)$ denotes a random quantity tending to zero as either $N$ or $M$ tends to infinity.
From here, we will omit the dependency on $j$ for the simplicity of notation when using the equivalent kernel.

\begin{remark}
The most important reason for using the Local-Polynomial fitting for the estimation of state variables and their derivatives is its rich literature in asymptotic properties and uniform convergence of the estimator~\cite{fan1997local,mack1982weak,tusnady1977remark,fan2018local}.
Specifically, these results allow us to explore the behavior of the tail-probability of the measurement error $\btau$, which is essential for the analysis of the $\ell_{1}$-PsLS estimator.
See Subsection~\ref{Subsec5.2} for more information. 
\end{remark}

\subsection{$\ell_1$-regularized Pseudo Least Square Model} \label{method_intro}
First, we introduce matrix-vector notations for compact expressions of the problem. 
We let $\mathbf{u}_t\in\mathbb{R}^{NM}$ denote the vectorization of $\{u_t(X_i,t_n)\}_{i=0,\dots,M-1}^{n=0,\dots,N-1}$ in a dictionary order prioritizing the spatial dimension; that is, $\mathbf{u}_t^T=\begin{bmatrix}
u_t(X_0,t_0)&u_t(X_1,t_0)&\cdots	
\end{bmatrix}
$. Define the \textit{feature matrix}, $\mathbf{F}\in\mathbb{R}^{NM\times K}$, as the collection of values of feature variables organized as follows:
\begin{align*}
\mathbf{F} := \begin{bmatrix}
 	1&u(X_0,t_0) & \partial_xu(X_0,t_0)  &\cdots& \partial^p_xu(X_0,t_0)\partial_x^qu(X_0,t_0) &\cdots\\
 	1&u(X_1,t_0) & \partial_xu(X_1,t_0)  &\cdots& \partial^p_xu(X_1,t_0)\partial_x^qu(X_1,t_0) &\cdots\\
 	\vdots&\vdots& \vdots&\ddots&\vdots&\cdots\\
 	1&u(X_{M-1},t_0) & \partial_xu(X_{M-1},t_0)  &\cdots& \partial^p_xu(X_{M-1},t_0)\partial_x^qu(X_{M-1},t_0) &\cdots\\
 	1&u(X_{0},t_1) & \partial_xu(X_{0},t_1)  &\cdots& \partial^p_xu(X_{0},t_1)\partial_x^qu(X_{0},t_1) &\cdots\\
 	\vdots&\vdots& \vdots&\ddots&\vdots&\cdots\\
 	1&u(X_{M-1},t_{N-1})&\partial_xu(X_{M-1},t_{N-1})&\cdots& \partial^p_xu(X_{M-1},t_{N-1})\partial_x^qu(X_{M-1},t_{N-1})  &\cdots
 	\end{bmatrix}\;.
\end{align*}
With these notations, the equation \eqref{eq.PDE1} can be written as $\bu_{t}=\bF\bbeta^*$.
Note that before estimating the correct signed-support of $\bbeta^*$, $\bu_{t}$ and $\bF$ need to be estimated.
Conventional regression techniques such as Local-Polynomial regression, smoothing spline, and regression spline,
among others, can be used to estimate $\bu_{t}$ and columns of $\bF$.
In this paper, we employ the Local-Polynomial approach.
We denote $\widehat{\mathbf{u}}_t\in\mathbb{R}^{NM}$ and $\widehat{\mathbf{F}}\in\mathbb{R}^{NM\times K}$ by replacing the entries of $\mathbf{u}_t$ and $\mathbf{F}$ respectively with the corresponding estimators. (i.e., $\widehat{(u_t)_i^n}$, $\widehat{(\partial_x^pu)_i^n}$, and $\widehat{(\partial_x^pu)_i^n}\widehat{(\partial_x^qu)_i^n}$.) 

Let $\Delta \bu_t=\widehat{\bu}_t-\bu_t$, $\Delta \bF=\widehat{\bF}-\bF$ denote the difference between the obtained estimators $\widehat{\bu}_t$ and $\widehat{\bF}$ via Local-Polynomial regression and their ground-truth counterparts.
With these notations, we formally obtain a regression model
\begin{align} \label{residual}
    \widehat{\bu}_t=\widehat{\bF}\bbeta^*+\btau\;,\;\;\;\text{where}~\btau=\Delta \bF\bbeta^*-\Delta \bu_t\;.
\end{align}
The natural extension for inducing sparsity of the parameter of interest is to add positively weighted $\ell_{1}$-penalty term $\left\|\bbeta\right\|_{1}$ to the squared loss $\| \widehat{u}_{t}-\widehat{F}\bbeta \|_{2}^{2}$, which leads to the following estimator:
\begin{align} \label{eq.lassomat}
    \widehat{\bbeta}^\lambda=\arg\min_{\bbeta\in\mathbb{R}^K}
    \bigg\{\frac{1}{2NM}\left\|\widehat{\bu}_t-\widehat{\bF}\bbeta\right\|_{2}^{2}+\lambda_{N}\left\|\bbeta\right\|_1\bigg\}\;,
\end{align}
where $\lambda_{N}>0$ is a regularization hyper-parameter.
Note that we normalize the columns of $\widehat{\bF}$ such that $\frac{1}{\sqrt{NM}} \max_{j=1,\dots,K} \| \widehat{\bF}_{j} \|_{2} \leq 1$ while solving \eqref{eq.lassomat}.

Observe that~\eqref{eq.lassomat} is formally identical to Lasso \cite{tibshirani1996regression} for high-dimensional sparsity recovery.
We call \eqref{eq.lassomat} as \textbf{$\ell_{1}$-Pseudo Least Square} method instead of Lasso.
Similarly with~\cite{liang2008parameter}, the word \textbf{\textit{pseudo}} comes from the setting of our problem, that is, $\widehat{\bbeta}^{\lambda}$ is not a true $\ell_{1}$-least square estimator, but a minimizer of the $\ell_{1}$-least square fit with the estimated $\widehat{\bu}_{t}$ and $\widehat{\bF}$.

Additionally, the residual vector $\btau$ violates conventional assumptions on residuals in linear regression, where they are assumed to be centered and independent among entries. 
See \cite{zhao2006model,wainwright2009sharp,knight2000asymptotics}.
If $\widehat{\bu}_{t}$ and $\widehat{\bF}$ are unbiased estimators of $\bu_{t}$ and $\bF$, $\btau$ is a residual vector with mean zero, but its entries are not independent.
However, if $\widehat{\bu}_{t}$ and $\widehat{\bF}$ are biased estimators such as Local-Polynomial estimators in our case, $\btau$ is not a mean zero random vector.
Moreover, the unknown signal $\bbeta^*$ makes the distribution of $\btau$ completely inaccessible.
These complexities make the study of the proposed estimator $\widehat{\bbeta}^{\lambda}$ challenging.

\section{Recovery Theory for $\ell_{1}$-PsLS based PDE Identification} \label{Sec4}
In Subsection \ref{ssec4.1}, we formally describe a signed-support recovery problem.
In Subsection \ref{ssec4.2}, two regularity assumptions on feature matrix $\bF$ are given for the proof of the main theorem. 
Then, the main theorem of this work is presented with some important remarks in Subsection \ref{ssec:main}.
Lastly, we provide a proof sketch of the main Theorem in Subsection \ref{proof_strat}.

\subsection{Signed-Support Recovery} \label{ssec4.1}
The main goal of this paper is to provide provable guarantees that the proposed $\ell_{1}$-PsLS method gives asymptotically consistent estimator of $\bbeta^*$ in the sense of signed-support recovery.
We can formally state this problem with the adoption of $\mathbb{S}_{\pm}(\bbeta)$ notation, that is: 
for any vector $\bbeta\in\mathbb{R}^{K}$, we define its extended sign vector, whose each entry is written as: 
\begin{align*}
    \mathbb{S}_{\pm}(\beta_{i}):=
    \begin{cases}
        +1   & \text{if } \beta_{i} > 0 \\
        -1   & \text{if } \beta_{i} < 0 \\
         0   & \text{if } \beta_{i} = 0,
    \end{cases}
\end{align*}
for $i\in\{1,\dots,K\}$.
This notation encodes the \textit{signed-support} of the vector $\bbeta$.
Denote $\widehat{\bbeta}^{\lambda}$ as the unique solution of $\ell_{1}$-PsLS. 
Then under some regularity conditions on the design matrix $\bF$, we will show,
\begin{equation*}
    \mathbb{P}\big[\mathbb{S}_{\pm}(\widehat{\bbeta}^{\lambda})=\mathbb{S}_{\pm}(\bbeta^{*})\big] \rightarrow{1} \quad \text{as} ~N, M\rightarrow{+\infty},
\end{equation*}
where $N$ and $M$ denote the grid size of temporal and spatial dimensions, respectively.

\subsection{Assumptions} \label{ssec4.2}
We introduce two sufficient conditions frequently assumed in $\ell_{1}$-regularized regression models for the signed-support recovery of the true signal $\bbeta^*$.
\begin{enumerate}
    \item \textbf{Minimal eigenvalue condition.} There exists some constant $C_{\min}>0$ such that:
    \begin{align}
        \Lambda_{\min}\bigg( \frac{1}{NM} \bF_{\mathcal{S}}^{T}\bF_{\mathcal{S}} \bigg) 
        \geq C_{min}.\label{assum1}\tag{A1}
    \end{align}
    Here $\Lambda_{\min}(\textbf{A})$ denotes the minimal eigenvalue of a square matrix $\textbf{A}\in\mathbb{R}^{n \times n}$, and $\bF_{\mathcal{S}}$ is made of columns of $\bF$ when the column index is in the support set $\mathcal{S}$.
    Note that if this condition is violated, the columns of $\bF_{\mathcal{S}}$ would be linearly dependent, and it would be impossible to estimate the true signal $\beta^*$ even in the ``oracle case'' when the support set $\mathcal{S}$ is known a priori.
    
    \item \textbf{Mutual incoherence condition.} For some \textit{incoherence parameter} $\mu\in(0,1]$:
    \begin{align}
        \left\VERT \big( \bF_{\mS^{c}}^{T}\bF_{\mS} \big)\big( \bF_{\mS}^{T}\bF_{\mS} \big)^{-1} \right\VERT_{\infty} \leq 1-\mu.\label{assum2}\tag{A2}
    \end{align}
    This condition states that the irrelevant predictors cannot exhibit an overly strong influence on the relevant predictors.
    More specifically, for each index $j\in\mS^{c}$, the vector  $(\bF_{\mS}^T\bF_\mS)^{-1}\bF_\mS^T\bF_{j}$ is the regression coefficient of $\bF_{j}$ on $\bF_\mS$, thus, 
    it is a measure of how well the column $\bF_j$  aligns with the columns of $\bF_\mS$.
    A large $\mu$ close to $1$ indicates that the columns $\{\bF_j, j \in \mS^{c}\}$ are nearly orthogonal to the columns of $\bF_\mS$, which is desirable for support recovery. 
\end{enumerate}

For future reference, we define $\mathcal{Q}^* := \big( \bF_{\mS^{c}}^{T}\bF_{\mS} \big)\big( \bF_{\mS}^{T}\bF_{\mS} \big)^{-1}$, 
and name it as \textit{population incoherence matrix}.
Also, define its estimated counterpart as $\widehat{\mathcal{Q}}_{N}:=\big( \widehat{\bF}_{\mS^{c}}^{T}\widehat{\bF}_{\mS} \big)\big( \widehat{\bF}_{\mS}^{T}\widehat{\bF}_{\mS} \big)^{-1}$, and call it \textit{sample incoherence matrix}.
Note that the dependence of the support set $\mathcal{S}$ on quantities $\mathcal{Q}^*$ and $\widehat{\mathcal{Q}}_{N}$ is suppressed for notational simplicity.

\subsection{Statement of Main Result} \label{ssec:main}
\begin{theorem} \label{thm1}
Given the observed data set $\mathcal{D}$ whose spatial resolution is related to the temporal resolution via $M = \Theta(N^{\frac{2P_{\max}+5}{7}})$, we take the bandwidths of the kernels in~\eqref{eq_K_b} and~\eqref{eq_K_c} as $h_N=\Theta(N^{-\frac{1}{7}})$,  $w_{M} = \Theta(M^{-\frac{1}{7}})$, respectively.
Under the assumptions \eqref{assum1} and \eqref{assum2} imposed on the ground-truth feature matrix $\bF$, suppose that the sequence of regularization hyper-parameters $\{\lambda_{N}\}$ satisfies 
$\lambda_{N} = \Omega{ \bigg( \frac{C\sqrt{K}\ln N}{\mu N^{2/7-c}} \bigg)}$ for some large enough $N$, some constant $C>0$ and $0<c<\frac{2}{7}$ independent of $N$. 
Then, the following properties hold with probability greater than $1-\mathcal{O}\left(N^{{\frac{2P_{\max}+5}{7}}}\exp\big(-\frac{1}{6}N^{c}\big) \right)\to 1$ as $N\rightarrow{\infty}$:
\begin{enumerate} [label=(\roman*)]
    \item The $\ell_{1}$-PsLS method \eqref{eq.lassomat} has a unique minimizer $\widehat{\bbeta}^{\lambda}\in\mathbb{R}^{K}$ with its support contained within the true support, that is $\mathcal{S}(\widehat{\bbeta}^{\lambda})\subseteq\mathcal{S}(\bbeta^*)$, and the estimator satisfies the $\ell_{\infty}$ bound:
    \begin{align}
        \left\|\widehat{\bbeta}_\mS^\lambda-\bbeta_\mS^*\right\|_\infty\leq K^{3/2}C_{\min}\left(C\frac{\ln N}{N^{2/7-c}}+\lambda_{N}\right)\;.\label{infbound}
    \end{align}
    \item Additionally, if the minimum value of the model parameters  supported on $\mathcal{S}$ is greater than the upper-bound of \eqref{infbound}, that is $\min_{1 \leq i \leq s}|(\bbeta_{\mathcal{S}}^*)_{i}|>K^{3/2}C_{\min}\left(C\frac{\ln N}{N^{2/7-c}}+\lambda_{N}\right)$, then $\widehat{\bbeta}^{\lambda}$ has a correct signed-support. i.e., $\mathbb{S}_{\pm}(\widehat{\bbeta}^{\lambda})=\mathbb{S}_{\pm}(\bbeta^{*})$.
\end{enumerate}
\end{theorem}

The overall proof sketch of Theorem \ref{thm1} is described in the Subsection \ref{proof_strat}, and relevant technical propositions and lemmas are further provided in Sections \ref{Sec4} and \ref{Sec5}.
Here, we give some important remarks about Theorem~\ref{thm1}.

\begin{enumerate}
    \item The uniqueness claim of $\widehat{\beta}^{\lambda}$ in (\rn{1}) seems trivial since the objective function in \eqref{eq.lassomat} is strictly convex in the regime of $K$ being fixed and $NM\rightarrow{\infty}$.
    However, we need to ensure that minimal eigenvalue condition hold over the estimated feature matrix $\widehat{\bF}$,
    given the assumption \eqref{assum1} for some $C_{\min}>0$.
    We defer this statement as Lemma~\ref{lem.eign'} in Section~\ref{Sec6}  and provide the proof in Appendix~\ref{proof_eigen'}.
    
    \item The first item (\rn{1}) claims that $\ell_{1}$-PsLS does not falsely select the arguments in  that are not in the support of $\bbeta^*$.
    Also note that part (\rn{2}) is a consequence of the sup-norm bound from \eqref{infbound}: as long as minimum value of $|\bbeta^*_{i}|$ over indices $i\in\mathcal{S}$ is not small, $\ell_{1}$-PsLS is signed-support recovery consistent.
    
    \item The asymptotic orders of $M$, $h_N$, and $w_M$ are specifically chosen for simplification. Although there is certain flexibility, the spatial resolution $M$ and the temporal resolution $N$ (as well as $h_N$ and $w_M$) need to be coordinated well to guarantee the support recovery property. 
    This was expected in practice since we need sufficient sampling frequencies both in temporal and space to estimate the underlying dynamics. 
    Here, the Theorem~\ref{thm1} present a rigorous justification for a combination of these resolutions which is sufficient for the support recovery. 

    \item The quantity $c$ is derived from the Tusnady's strong approximation~\cite{tusnady1977remark} where the error of an empirical distribution is compared with a Brownian bridge in tail probability. See Appendix~\ref{app.B1}. 
    With a larger value of $c$, the regularization hyper-parameter $\lambda_{N}$ needs to remain relatively large, but the convergence is faster. 
    Whereas for a smaller value of $c$, we can relax the regularization in the cost of a slower probability convergence rate.

    \item The threshold of $\lambda_{N}$ in the statement of the Theorem shows that when the number of data increases, there is more flexibility in tuning this parameter. If the incoherence parameter $\mu$ is small, or equivalently, the group of correct feature variables and the group of the others are similar, to guarantee that the support of the estimated coefficient vector is contained in the correct one, it suffices to use a large value of $\lambda_{N}$. 
    Such behavior of the threshold is consistent with that described in Theorem 1 of~\cite{wainwright2009sharp}.
    
    \item The upper-bound for the $\ell_\infty$-norm of the coefficient error in~\eqref{infbound} consists of two components. 
    The first one concerns the grid resolution determined by $N$, and the underlying function $u$ as well as the choice of regression kernels encapsulated in the constant $C$. 
    As $N$ increases to $\infty$, this part converges to $0$ without explicit dependence on the choice of feature variables selected by $\ell_1$-PsLS.
    The second component is simple: $K^{3/2}C_{\min}\lambda_{N}$. 
    When $N$ increases, this part does not vary.
    This indicates that asymptotically, $\ell_{1}$-PsLS recovers signed-support of governing PDE, as long as $\min_{1\leq i \leq s}|(\bbeta_{\mathcal{S}}^*)_{i}|>K^{3/2}C_{\min}\lambda_{N}$.
\end{enumerate}

\subsection{Proof Strategy of Theorem \ref{thm1}} \label{proof_strat}
The analysis for the proof of Theorem \ref{thm1} is naturally divided into two steps as follows: 
In the first step, we prove a result analogous to that of the Theorem \ref{thm1} by imposing incoherence assumption on the estimated feature matrix $\widehat{\bF}$. 
Specifically, since $\widehat{\bF}$ is a random matrix, we assume that for some $\mu\in(0,1]$, the event, $\{\VERT \widehat{\mathcal{Q}}_{N}\VERT_{\infty}\leq 1-\mu\}$, holds with some probability at least $P_{\mu}$, for some $P_{\mu}\in (0,1]$. 
Under this assumption, we prove that the success probability of signed-support recovery of $\ell_{1}$-PsLS converges to $P_{\mu}$ with an  exponential decay rate. 
This is formally stated as Proposition \ref{prop1} in Subsection \ref{state_prop1}.
\\ \\
In the second step, we show that the success probability $P_{\mu}$ goes to $1$, given that the ground-truth matrix $\bF$ satisfies assumptions \eqref{assum1} and \eqref{assum2}.
This is equivalent to proving that, given the assumptions \eqref{assum1} and \eqref{assum2} for $\bF$ for some $C_{\min}>0$ and $\mu\in(0,1]$, the same assumptions hold for the estimated $\widehat{\bF}$ in probability.
We state these results formally in Lemmas \ref{lem.eign'} and \ref{lem.incoh'} in Section \ref{Sec6}.

\section{Analysis Under Sample Incoherence Matrix Assumptions} \label{Sec5}
In this section, we provide a proof overview of Proposition~\ref{prop1} and the key technical contribution of our paper. 
All the detailed statements and proofs of the Proposition~\ref{prop1} and its relevant lemmas are relegated to the Appendix for the conciseness.

\subsection{Statement of Proposition} \label{state_prop1}
We establish the signed-support consistency of $\ell_{1}$-PsLS estimator when the assumptions are directly imposed on the estimated feature matrix $\widehat{\bF}$, instead on the ground-truth feature matrix $\bF$. 
More specifically, we assume that there exist some constants $\mu\in(0,1]$ and $C_{\min}>0$, such that the followings hold:
\begin{equation}\label{assum3}\tag{A3}
    \mathbb{P}\bigg[ \left\VERT \widehat{\mathcal{Q}}_{N} \right\VERT_{\infty} \leq 1-\mu\bigg]\geq P_\mu
    ~\text{ and }
    ~\Lambda_{\min}\Big(\frac{1}{NM}\widehat{\bF}_\mS^T\widehat{\bF}_\mS\Big)\geq C_{\min} \quad \text{almost surely}\;.
\end{equation}
Here, $P_{\mu}\in [0,1]$ denotes some probability that $\widehat{\mathcal{Q}}_{N}$ satisfies the incoherence assumption.
Equipped with this assumption, we have the following proposition:

\begin{proposition} \label{prop1}
Given the observed data set $\mathcal{D}$, where the spatial resolution is related to the temporal resolution via $M = \Theta(N^{\frac{2P_{\max}+5}{7}})$, we take the bandwidths of the kernels in~\eqref{eq_K_b} and~\eqref{eq_K_c} as $h_N=\Theta(N^{-\frac{1}{7}})$,  $w_{M} = \Theta(M^{-\frac{1}{7}})$, respectively.
Under the assumptions in \eqref{assum3} imposed on the estimated feature matrix $\widehat{\bF}$, suppose that the sequence of regularization hyper-parameters $\{\lambda_{N}\}$ satisfies $\lambda_{N} = \Omega{ \bigg( \frac{C\sqrt{K}\ln N}{\mu N^{2/7-c}} \bigg)}$ for some constant $C>0$ and $0<c<\frac{2}{7}$ independent of $N$. 
Then, the following properties hold : 
\begin{enumerate} [label=(\roman*)]
    \item  With probability greater than $P_{\mu}-\mathcal{O}\left(N^{{\frac{2P_{\max}+5}{7}}}\exp\big(-\frac{1}{6}N^{c}\big) \right)\to P_{\mu}$ as $N\rightarrow{\infty}$, the $\ell_{1}$-PsLS method \eqref{eq.lassomat} has a unique minimizer $\widehat{\bbeta}^{\lambda}\in\mathbb{R}^{K}$ with its support contained within the true support, that is $\mathcal{S}(\widehat{\bbeta}^{\lambda})\subseteq\mathcal{S}(\bbeta^*)$.
    
    \item With probability greater than $1-\mathcal{O}\left(N^{{\frac{2P_{\max}+5}{7}}}\exp\big(-\frac{1}{6}N^{c}\big) \right)\to 1$ as $N\rightarrow{\infty}$, $\widehat{\bbeta}^{\lambda}$ satisfies the $\ell_{\infty}$ bound:
    \begin{align}
        \left\|\widehat{\bbeta}_\mS^\lambda-\bbeta_\mS^*\right\|_\infty\leq K^{3/2}C_{\min}\left(C\frac{\ln N}{N^{2/7-c}}+\lambda_N\right)\;.\label{infbound1}
    \end{align}
    
    \item Additionally, if the minimum value of model parameter supported on $\mathcal{S}$ is greater than the upper-bound of \eqref{infbound1}, that is $\min_{1 \leq i\leq s}|(\bbeta_{\mathcal{S}}^*)_{i}|>K^{3/2}C_{\min}\left(C\frac{\ln N}{N^{2/7-c}}+\lambda_N\right)$, then $\widehat{\bbeta}^{\lambda}$ has a correct signed-support. (i.e., $\mathbb{S}_{\pm}(\widehat{\bbeta}^{\lambda})=\mathbb{S}_{\pm}(\bbeta^{*})$)
\end{enumerate}
\end{proposition}

We remark that the first item (\rn{1}) in Proposition \ref{prop1} holds with probability $P_{\mu}\leq 1$ asymptotically, while the second item (\rn{2}) holds with probability $1$ asymptotically.
They are not contradictory, since (\rn{1}) focuses on the estimation errors on entries within the true support $\mS$, whereas (\rn{2}) describes the support recovery of the coefficient vector over all indices.
Technically speaking, proof of (\rn{1}) is involved with mutual incoherence condition in \eqref{assum3}, whereas (\rn{2}) is involved with minimum-eigen value condition on $\widehat{\bF}$ in \eqref{assum3}.

\subsection{Proof Overview of Proposition ~\ref{prop1}} \label{Subsec5.2}
Readers can find the proof of \eqref{infbound1} in the Appendix~\ref{B6}.
Here, we focus on providing the high-level idea on the proof of (\rn{1}) of Propostion~\ref{prop1}.
The most important ingredient for the success of PDW construction is to establish the \textit{strict dual feasibility} of the dual vector $\widehat{z}$, when $\widehat{z}\in\partial\|\widehat{\bbeta}^{\lambda}\|_{1}$, where $\partial\|\widehat{\bbeta}^{\lambda}\|_{1}$ is a sub-differential set of $\|\cdot\|_{1}$ evaluated at $\widehat{\bbeta}^{\lambda}$.
In other words, we need to ensure that  $\|\widehat{z}_{\mathcal{S}^{c}}\|_{\infty}<1$ with high probability. 
(See Appendix \ref{app.pdw}.) 
Through Karush–Kuhn–Tucker (KKT) condition of the optimal pair $(\widehat{\bbeta}^{\lambda},\widehat{z})$ of \eqref{eq.lassomat} and settings of PDW construction, we can explicitly derive the expression of the dual vector $\widehat{z}$ supported on the complement of the support set $\mathcal{S}$ as follows:
\begin{align} 
    \widehat{\mathbf{z}}_{\mS^c}=\widehat{\bF}_{\mS^c}^T\widehat{\bF}_\mS(\widehat{\bF}_{\mS}^T\widehat{\bF}_\mS)^{-1}\mathbf{\widehat{z}}_\mS+
    \underbrace{\frac{1}{\lambda_{N} MN}\widehat{\bF}_{\mS^c}^T{\bf{\Pi}_{\mS^{\perp}}}(\Delta \bu_t-\Delta \bF_\mS\bbeta_\mS^*)}_{:=\Tilde{Z}_{\mathcal{S}^{c}}} \;,\label{zSc'}
\end{align}
where ${\bf{\Pi}_{\mS^{\perp}}}$ is an orthogonal projection operator on the column space of $\widehat{\bF}_\mS$.
By the mutual incoherence condition in \eqref{assum3}, the first term of the right-hand side in~\eqref{zSc'} is upper-bounded by $1-\mu$ for some $\mu\in(0,1]$, with some probability $P_{\mu} \in [0,1]$.
The remaining task is to control the tail probability of $\Tilde{Z}_{j}$ for $j\in\mathcal{S}^{c}$: that is to ensure $\mathbb{P}\big[ \max_{j\in\mathcal{S}^{c}}|\Tilde{Z}_{j}| \geq \mu \big]\rightarrow{0}$ with some exponential decay rate.
With the help of Lemma \ref{Zlemma} in the Appendix, controlling the probability $\mathbb{P}\big[ \|\Tilde{Z}_{\mathcal{S}^{c}}\|_{\infty} \geq \mu \big]$ reduces to controlling $\mathbb{P}\big[ \|\Delta \bF_\mS\bbeta_\mS^*-\Delta \bu_{t}\|_{\infty} \geq \mu\frac{\lambda_{N}}{\sqrt{K}} \big]$.
Controlling the bound on $\mathbb{P}\big[\|\btau\|_{\infty} \geq \varepsilon \big]$ for some $\varepsilon>0$ is challenging, since the exact form of the residual distribution $\btau$ is unknown. 
(Note that $\btau=\Delta \bF_\mS\bbeta_\mS^*-\Delta \bu_t$ since $\bu_{t}=\bF\bbeta^*$.)

We circumvent this difficulty by using the following inequality: 
for some thresholds $\varepsilon_{N}>0$ and $\varepsilon_{M}>0$, both of which go to $0$ as $N$ and $M$ tends to $\infty$, we have,
\begin{align*}
    \mathbb{P} &\big[ \left\| \btau \right\|_{\infty} \geq \varepsilon_{N}+\varepsilon_{M}  \big] \\
    &\leq \mathbb{P}\bigg[ \max_{0 \leq i \leq M-1}\sup_{t\in[0,T_{\max})}|\Delta u_t(X_i,t)| \geq \varepsilon_{N} \bigg] + 
    \mathbb{P}\bigg[ \max\limits_{1 \leq k \leq s \atop 0 \leq n \leq N-1 }\sup_{x\in[0, X_{\max})}|\Delta F_k(x,t_n)| \geq \frac{\varepsilon_{M}}{s\|\bbeta^*\|_{\infty}} \bigg] \\
    &\leq M\cdot\mathbb{P}\bigg[ \sup_{t\in[0,T_{\max})}|\Delta u_t(X_i,t)| \geq \varepsilon_{N} \bigg] + 
    sN\cdot\mathbb{P}\bigg[ \sup_{x\in[0, X_{\max})} |\Delta F_k(x,t_n)| \geq \frac{\varepsilon_{M}}{s\|\bbeta^*\|_{\infty}} \bigg].
\end{align*}
The above inequality naturally leads us to study the uniform convergence of Local-Polynomial estimator to its ground-truth function of interest. 
Say, for sufficiently large enough grid size of temporal dimension $N$, for some $\varepsilon_{N}\geq 0$ that is $h_{N}$-dependent threshold and $X_{i}\in[0,X_{\max})$, we will achieve
\begin{equation} \label{unif_lp}
    \mathbb{P}\Bigg[ \sup_{t\in[0,T_{\max})}\left| \widehat{\bu}_t(X_i,t) - \bu_t(X_i,t) \right| > \varepsilon_{N} \Bigg] \rightarrow{0},
\end{equation}
with an exponential decay rate.
As for obtaining the exponential decay rate in \eqref{unif_lp}, we defer the detailed explanation with some intuitions in the following Subsection.
It turns out that thresholds $\varepsilon_{N}$ and $\varepsilon_{M}$ are functions of bandwidth parameters $h_{N}$ and $w_{M}$ in \eqref{LP_t} and \eqref{LP_x}. 
We choose correct orders of $h_{N}$ and $w_{M}$ so that we can ensure that the thresholds $\varepsilon_{N}$ and $\varepsilon_{M}$ go to zero. Then, with the proper choice on the order of $\lambda_{N}$ together with $\mathbb{P}\big[ \|\btau\|_{\infty} \geq \mu\frac{\lambda_{N}}{\sqrt{K}} \big]$, we conclude the proof.

\subsection{Technical Contribution}
Several researchers have tried to achieve uniform convergence of Local-Polynomial or kernel smoothing estimators in almost sure sense. 
See the works of Masry~\cite{masry1996multivariate} and Li and Hsing~\cite{li2010uniform}. 
However, to the best of the authors' knowledge, uniform convergence of Local-Polynomial estimator with an explicit decaying probability rate has not been studied in the literature.
We provide it as a technical contribution of the present paper.
Readers can find the exact statements of these results for the estimators $\widehat{\bu}_{t}$ and $\widehat{\partial_{x}^{p}u}$ for $p\geq 0$ in the Appendix B.1 and B.2 stated as Lemma~\ref{Lemm.Dt} and Lemma~\ref{Lemm.Dx}, respectively.

Here, we provide a high-level idea of the proof of Lemma~\ref{Lemm.Dt}.
First, we observe that the higher-order Local-Polynomial smoothing is asymptotically equivalent to higher-order kernel smoothing through equivalent kernel theory~\cite{fan1997local}.
See \eqref{LP_t} and \eqref{LP_x} for their equivalences in mathematical form with kernel smoothing estimators.
Second, we employ Mack and Silverman's \cite{mack1982weak} truncation idea on the Local-Polynomial estimator and decompose $\widehat{\bu}_t(X_i,t) - \bu_t(X_i,t)$ into three parts as follows:
\begin{align*}
    \widehat{\bu}_{t}-\bu_{t} 
        = \underbrace{\bigg( \widehat{\bu}_{t} - \widehat{\bu_t}^{B_N'} - \mathbb{E} \big( \widehat{\bu}_{t} - \widehat{\bu_t}^{B^{'}_N} \big) \bigg)}_{\text{ Asymptotic deviation of truncation error}}
        + \underbrace{\bigg( \widehat{\bu_t}^{B^{'}_N} - \mathbb{E} \widehat{\bu_t}^{B^{'}_N} \bigg)}_{\substack{\text{Asymptotic deviation of} \\ \text{truncated estimator}}}
        + \underbrace{\bigg( \mathbb{E}\widehat{\bu}_{t} - \bu_{t} \bigg)}_{\text{Asymptotic bias}},
\end{align*}
where $B'_{N}$ is some increasing sequence in $N$, and $\widehat{\bu_t}^{B_N'}$ denotes the truncated Local-Polynomial estimator of $\bu_{t}$. 
We control the \emph{sup} over $t\in[0,T_{\max})$ on each of these three components.
The last component, \textit{Asymptotic bias} of $\widehat{\bu}_{t}$ can be obtained through the classical result from \cite{fan1997local,fan2018local}.
The exponential decay rate comes from the first two components as follows:

\begin{enumerate}
    \item \textit{Asymptotic deviation of truncation error} can be decomposed into two parts. 
    The first part, which is $\widehat{\bu}_{t} - \widehat{\bu_t}^{B_N'}$, can be easily controlled via chernoff bound of gaussian random variable. by using the definition of truncated estimator $\widehat{\bu_t}^{B_N'}$.
    The second part, which is the expected difference $\mathbb{E} \big( \widehat{\bu}_{t} - \widehat{\bu_t}^{B^{'}_N} \big)$, can be bounded by some deterministic function of $B_{N}'$ and $h_{N}$ using the similar arguments in Proposition $1$ of \cite{mack1982weak}.
    
    \item \textit{Asymptotic deviation of truncated estimator} is decomposed into two components as well: (\rn{1}) Brownian Bridge and (\rn{2}) difference between some two-dimensional empirical process and the Brownian Bridge.
    (\rn{1}) can be controlled via uniform convergence of Gaussian Process using the arguments similar to \cite{silverman1978weak}, together with simple Markov inequality.
    (\rn{2}) can be controlled via Tusnady's strong uniform approximation theory \cite{mack1982weak,tusnady1977remark}, stating that the two-dimensional empirical process can be well approximated by a certain solution path of two-dimensional Brownian-bridge.
\end{enumerate}
Same ideas can be employed for the uniform convergence of $\widehat{\partial_x^pu}$ to $\partial_x^p u$ for $p\geq 0$.

\section{Uniform Convergence of Sample Incoherence Matrix} \label{Sec6}
In this section, we provide two lemmas that can complete the proof of Theorem \ref{thm1}.
Here, the minimum-eigenvalue and incoherence assumptions are imposed on the ground-truth feature matrix $\bF$, instead on the estimated feature matrix $\widehat{\bF}$.
See (\ref{assum1}) and (\ref{assum2}).
That is, there exist $C_{\min}>0$ and $\mu\in(0,1]$ such that the followings hold for the unknown support set $\mS$:
\begin{equation*}
     \Lambda_{\min}\Big(\frac{1}{NM}{\bF}_\mS^T{\bF}_\mS\Big)\geq C_{\min} \quad \text{and}
     \quad 
     \left\VERT \mathcal{Q}^* \right\VERT_{\infty}\leq 1-\mu.
\end{equation*}
Equipped with the above assumptions, we can formally show that success probability of the sample incoherence condition $P_{\mu}$ in (\ref{assum3}) tends to $1$ as $N\rightarrow{\infty}$.
Note that this result is not an immediate consequence of classical random matrix theory (see \cite{tao2012topics, davidson2001local}), since the elements in $\widehat{\bF}^{T}\widehat{\bF}$ are highly dependent. 

To prove the result, we first need the following lemma asserting that if there exists $C_{\text{min}}>0$ such that the minimum eigen-value condition holds for $\bF_{\mS}$, then the sample minimum eigen-value condition holds with probability converging to $1$ with an exponential decay rate.

\begin{lemma}\label{lem.eign'}
 Suppose that the assumption (\ref{assum1}) holds with some constant $C_{\text{min}}>0$ and $0<c<\frac{2}{7}$, then with probability at least
 $1-\mathcal{O}(N\exp(-\frac{1}{6}N^c))$, we have,
 \begin{equation*}
     \Lambda_{\min}\Big(\frac{1}{NM}\widehat{\bF}_\mS^T\widehat{\bF}_\mS\Big)\geq C_{\min}\;.
 \end{equation*}
\end{lemma}

\noindent With the help of Lemma \ref{lem.eign'}, we can show that the sample incoherence condition holds with high probability, given that there exists $\mu\in(0,1]$ for the ground-truth version of (\ref{assum2}).

\begin{lemma}\label{lem.incoh'}
 Suppose that the assumption (\ref{assum2}) holds with some constant $\mu\in(0,1]$ and $0<c<\frac{2}{7}$, then with probability at least
 $1-\mathcal{O}(N\exp(-\frac{1}{6}N^c))$, we have,
 \begin{equation*}
     \left\VERT \widehat{\mathcal{Q}}_{N} \right\VERT_{\infty}\leq 1-\mu\;.
 \end{equation*}
\end{lemma}

Verification of Lemma \ref{lem.incoh'} automatically leads to the complete proof of Theorem \ref{thm1}, together with Proposition \ref{prop1}.
Therefore, as long as the two assumptions (\ref{assum1}) and (\ref{assum2}) hold for $\bF$, with sufficiently fine-grained grid points over the function $u(X,t)$, $\ell_1$-PsLS can always find the correct signed-support of the given PDE model, with the minimum absolute value of $\bbeta^*_{\mathcal{S}}$ not too close to zero.

\begin{remark}
\noindent\textbf{(Technical Difficulties of Lemma \ref{lem.eign'} and \ref{lem.incoh'}.)} 
The proof procedure is involved with controlling the tail probability of difference between inner-product of two arbitrary columns of $\widehat{\bF}$ and inner-product of the two corresponding columns of ground-truth $\bF$. 
This problem is challenging even if the exact distribution of any entries of $\widehat{\bF}$ is known, since the distribution of $\sum_{k=1}^{NM}\widehat{F}_{ki}\widehat{F}_{kj}$ needs to be derived.
We circumvent this problem by taking the advantage of the uniform convergence result of $\widehat{\partial_x^pu}$ for any $p\geq 0$. 
\end{remark}

\section{Numerical Experiments} \label{sec_num}
In the first subsection, two PDE models and data-generating processes of respective models are introduced.
In the next subsection, we verify the main statements of the Theorem \ref{thm1} through numerical experiments over the PDE models described in Subsection \ref{sec_expsetting}. 
The impact of $\bbeta^*_{\min}$-condition in the signed-support recovery of $\ell_{1}$-PsLS is numerically explored in subsection \ref{ssec7.3}.

\subsection{Experimental Setting} \label{sec_expsetting}
In this subsection, we provide detailed descriptions on 
(\rn{1}) two popular PDE models that we are going to work on throughout the Section \ref{sec_num}, 
and on (\rn{2}) how to generate the data from respective models, and (\rn{3}) how to design the regression problem for the experiments to be presented.

\subsubsection{Model Specification and Data Generation}\label{subsec.PDEmodels}
\noindent\textbf{Viscous Burgers' equation} is a fundamental second-order semilinear PDE which is frequently employed to model physical phenomena in fluid dynamics~\cite{bonkile2018systematic} and nonlinear acoustic in dissipative media~\cite{rudenko1975theoretical}. 
Its general form is
$$u_t = -uu_x + \nu u_{xx}$$
where $\nu>0$ is the diffusion coefficient which characterizes physical quantities such as viscosity of fluid. 
Specifically, when $\nu=0$, it becomes an inviscid Burgers' equation, which is a conservative system that can form shock waves.
Here we consider the following viscous Burgers' equation: 
\begin{align}
    &u_t = -uu_x + \nu u_{xx}\;,\;\; 0<x<1, 0<t<0.1\label{eq_visBurgers}\\
    &u(x,0) = \sin^2(2\pi x)+\cos^3(3\pi x)\;,\;\;0\leq x\leq 1\;,\;\; u(0,t)=u(1,t)\;,\;\;0\leq t\leq 0.1.\nonumber
\end{align}

\noindent\textbf{Korteweg–de Vries equation} is well known for its soliton solutions that demonstrate the phenomenon of superposition of nonlinear waves~\cite{sawada1974method}, and for modeling fluid dynamics of shallow water surfaces in long and narrow channels~\cite{boussinesq1877essai}.
Its dimensionless form is given as 
\begin{align}\label{KdV}
    &u_t +u_{xxx}+6uu_{x}=0\;.
\end{align}
In this Section, we consider the form of \eqref{KdV}, whose initial solution is as follows:
\begin{align*}
    &u(x,0) = 3.5\sin^3(4\pi x)+1.5\exp\big(-\sin(2\pi x)(1-x)\big) \;,\;\; \\
    &0\leq x\leq 1\;,\;\; u(0,t)=u(1,t)\;,\;\;0\leq t\leq 0.1.\nonumber
\end{align*}
\noindent\textbf{Data Generation}
For $N$-size sampling in the temporal dimension, by Theorem~\ref{thm1}, we take $M=\lfloor N^{(2\times P_{\max}+5)/7} \rfloor$ sample size in the space dimension.  
We numerically solve Viscous Burgers' equation~\eqref{eq_visBurgers} by the Lax-Wendroff scheme on a grid with interval width $\delta t = 0.1/(100N)$ in temporal and  $\delta x =1/M $ in space, then we downsampled the data in the temporal dimension by a factor of $100$; thus the resulted clean data is distributed over a grid with $N$ nodes in temporal and $M$ nodes in space. 
Lastly, we added i.i.d. Gaussian noise with standard deviation $\sigma=0.25$ to the data. i.e., $\nu_{i}^{n}\overset{\text{i.i.d.}}{\sim}\mathcal{N}(0,0.25^2)$.
As for solving the KdV equation~\eqref{KdV}, the same approaches with Viscous Burger's equation are applied, with i.i.d. Gaussian noises with standard deviation $\sigma=0.025$.

\subsubsection{Constructions of Regression Problems} 
We employ the Local-Polynomial smoothing for estimating $\widehat{\bu}_{t}$ and $\widehat{\bF}$ as described in Subsection~\ref{local}.
Regarding a choice of kernel for constructing $\widehat{\bu}_{t}$ and $\widehat{\bF}$, we use the Epanechnikov kernel defined by:
\begin{align*}
	\mathcal{K}(z) = \frac{3}{4}(1-z^2)_+\;,~z\in\mathbb{R}\;,
\end{align*}
where $(\cdot)_+:=\max( 0,\cdot )$.
Bandwidth parameters $h_{N}$ and $w_{M}$ in \eqref{eq_K_b} and \eqref{eq_K_c} are chosen in the order of $h_N=\Theta(N^{-\frac{1}{7}})$ and $w_{M} = \Theta(M^{-\frac{1}{7}})$, respectively.
As displayed in Table~\ref{table1}, for the experiments presented in this Section, we choose specific constant factors in the order expressions of $h_N$ and $w_M$ for Viscous Burgers equation and KdV equation.
Regarding more detailed issues on the choices of these constants, readers can refer to Section~\ref{sec_conclusion}.
It is also worth noting that we do not use \eqref{LP_t} and \eqref{LP_x} as solutions of the optimization problems \eqref{eq_K_b} and \eqref{eq_K_c} for the experiments, since the expressions in \eqref{LP_t} and \eqref{LP_x} are derived in asymptotic settings.
For the reader's convenience, We provide the closed form solutions of \eqref{LP_t} and \eqref{LP_x} in Appendix~\ref{app.loc}.

For Viscous Burgers' equation, on the set of noisy data, Local-Polynomial fitting with $P_{\max}=2$ is applied to construct $\widehat{\bu}_{t}$ and $\widehat{\bF}$.
Specifically, our goal is to identify the fifth and the sixth coefficients, $\bbeta_{5}$ and $\bbeta_{6}$, of a following linear measurement via our proposed $\ell_1$-PsLS model~\eqref{eq.lassomat}:
\begin{align*}
    \widehat{\bu}_{t}=\bbeta_{0}+\bbeta_{1}\widehat{\bu}+\bbeta_{2}\widehat{\bu}^{2}+\bbeta_{3}\widehat{\bu}_{x}+\bbeta_{4}\widehat{\bu}_{x}^{2}+\bbeta_{5}\widehat{\bu}\widehat{\bu}_{x}+\bbeta_{6}\widehat{\bu}_{xx}+\bbeta_{7}\widehat{\bu}_{xx}^{2}+\bbeta_{8}\widehat{\bu}_{x}\widehat{\bu}_{xx}+\bbeta_{9}\widehat{\bu}\widehat{\bu}_{xx}.
\end{align*}

\begin{table}[t]
  \centering
  \begin{tabular}{l|l|l}
     & $w_{M}$  &  $h_{N}$ \\ \hline
     Viscous Burgers & $0.75 M^{-\frac{1}{7}}$ & $0.25 N^{-\frac{1}{7}}$ \\ \hline
     KdV & $0.1 M^{-\frac{1}{7}}$  & $0.01 N^{-\frac{1}{7}}$ 
  \end{tabular}
  \caption{ Specific choices of the constants in the order of $h_N=\Theta(N^{-\frac{1}{7}})$ and $w_{M} = \Theta(M^{-\frac{1}{7}})$ for the experiments on Viscous Burgers equation and KdV equation are presented.}
  \label{table1}
\end{table}

For KdV equation, after generating the data-points, $\widehat{\bu}_{t}$ and $\widehat{\bF}$ are fitted through Local-Polynomial with $P_{max}=3$.
We want $\ell_{1}$-PsLS to select $\bbeta_{5}$ and $\bbeta_{10}$ as non-zero coefficients in a following linear measurement:
\begin{align*}
    \widehat{\bu}_{t}=\bbeta_{0}+\bbeta_{1}\widehat{\bu}+\bbeta_{2}\widehat{\bu}^{2}+\bbeta_{3}\widehat{\bu}_{x}+&\bbeta_{4}\widehat{\bu}_{x}^{2}+\bbeta_{5}\widehat{\bu}\widehat{\bu}_{x}+\bbeta_{6}\widehat{\bu}_{xx}+\bbeta_{7}\widehat{\bu}_{xx}^{2}+\bbeta_{8}\widehat{\bu}_{x}\widehat{\bu}_{xx}+\bbeta_{9}\widehat{\bu}\widehat{\bu}_{xx}\\
    &+\bbeta_{10}\widehat{\bu}_{xxx}+\bbeta_{11}\widehat{\bu}_{xxx}^{2}+\bbeta_{12}\widehat{\bu}_{x}\widehat{\bu}_{xxx}+\bbeta_{13}\widehat{\bu}_{xx}\widehat{\bu}_{xxx}+\bbeta_{14}\widehat{\bu}\widehat{\bu}_{xxx}.
\end{align*}

\subsection{Numerical Verifications of Main Statements}
In this subsection, we design an experiment to numerically verify following two main statements of this paper.
\begin{enumerate}
    \item \textit{Under the assumptions (\ref{assum1}) and (\ref{assum2}), and with large enough data points, there exist some $\lambda\geq 0$ such that $\ell_1$-PsLS model~\eqref{eq.lassomat}  recovers a signed-support $\big( \mathbb{S}_{\pm}(\widehat{\bbeta})=\mathbb{S}_{\pm}(\bbeta^{*}) \big)$ of an unique PDE that admits the underlying function as a solution in probability.}
    \item \textit{Given the assumptions (\ref{assum2}) for some $\mu\in(0,1]$, sampled incoherence parameter $\mu'$ converges to ground-truth incoherence parameter $\mu$ in probability with large enough data points.}
\end{enumerate}

The experiment is conducted over two PDE models, \textbf{Viscous Burgers' equation} and \textbf{KdV equation} introduced in Subsection \ref{sec_expsetting}.
We generate the data by setting $\nu=0.03$ in \eqref{eq_visBurgers}.
In Figure \ref{fig_rate_zsc}, the probability of signed-support recovery $\mathbb{P}[\mathbb{S}_{\pm}(\widehat{\bbeta})=\mathbb{S}_{\pm}(\bbeta^{*})]$ versus the grid size of temporal dimension $N$, and $\left\|\widehat{z}_{\mathcal{S}^{c}}\right\|_{\infty}$ versus $N$ are recorded on the same plot for respective models.
Each point on each curve, which represents $\mathbb{P}[\mathbb{S}_{\pm}(\widehat{\bbeta})=\mathbb{S}_{\pm}(\bbeta^{*})]$, in (a) and (b) corresponds to the average over $100$ trials.
For each iteration, the hyper-parameter $\lambda_{N}$ is chosen in an ``optimal'' way: we used the value yielding the correct number of nonzero coefficient.
With the chosen $\lambda_{N}$, $\widehat{z}_{\mathcal{S}^{c}}$ is calculated as given in \eqref{zSc'}.
Note that \eqref{zSc'} can be calculated only when the $\ell_{1}$-PsLS finds $\lambda_{N}$ that gives the minimizer of \eqref{eq.lassomat} $\widehat{\bbeta}^\lambda$ such that $\widehat{\bbeta}_{\mathcal{S}^{c}}^\lambda=0$ and $\mS(\widehat{\bbeta}^\lambda)\subseteq\mS(\bbeta^*)$.
For this reason, boxplots of $\left\|\widehat{z}_{\mathcal{S}^{c}}\right\|_{\infty}$ in (a) and (b) are drawn from the point when $\ell_{1}$-PsLS starts to find such $\lambda_{N}$.
For both models, $\mathbb{P}[\mathbb{S}_{\pm}(\widehat{\bbeta})=\mathbb{S}_{\pm}(\bbeta^{*})]$ goes to $1$, as we observe more data points on finer grid.
Furthermore, it is worth noting that the \textit{strict dual feasibility} condition (i.e., $\left\|\widehat{z}_{\mathcal{S}^{c}}\right\|_{\infty}<1$) holds for both cases.
\begin{figure}
  \begin{tabular}{cc}
      (a){ Viscous Burgers} & (b){ KdV} \\
      \includegraphics[width=0.5\textwidth]{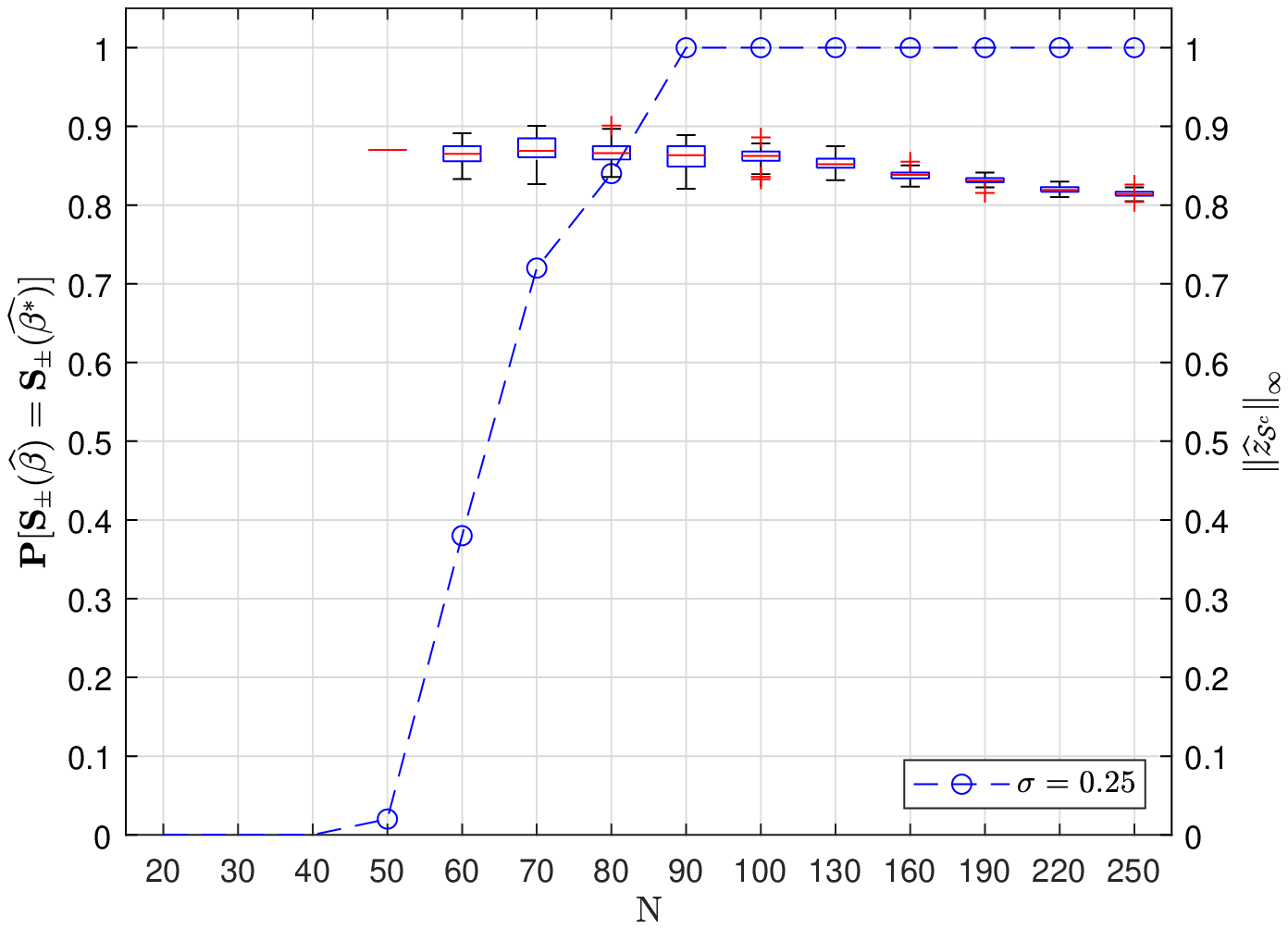}&
      \includegraphics[width=0.5\textwidth]{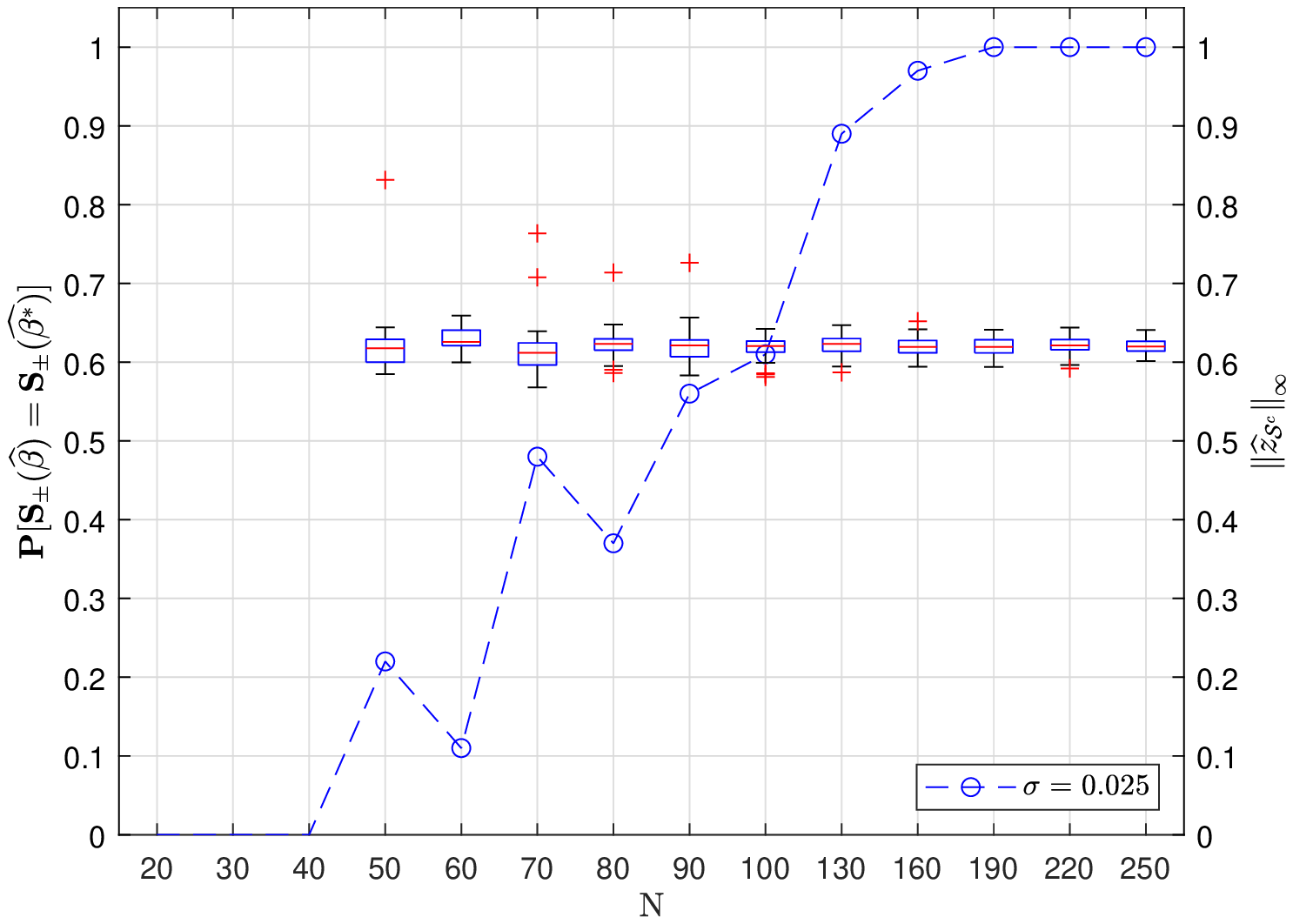}
  \end{tabular}
  \caption{Probability of signed-support recovery $\mathbb{P}[\mathbb{S}_{\pm}(\widehat{\beta})=\mathbb{S}_{\pm}(\beta^{*})]$ versus the grid size of temporal dimension $N$, and $\left\|\widehat{z}_{\mathcal{S}^{c}}\right\|_{\infty}$ versus $N$ are recorded on the same plot for Viscous Burger's equation in panel (a) and for KdV equation in panel (b), respectively.} 
  \label{fig_rate_zsc}
\end{figure}
In Figure \ref{fig_Inco}, boxplots of $\VERT \widehat{\mathcal{Q}}_{N}\VERT_{\infty}$ versus $N$ are displayed for Viscous Burgers' equation and KdV equation respectively.
A dotted horizontal line in each panel represents $1-\mu$ calculated from the ground-truth feature matrix $\bF$.
Notice that as the number of observed data gets larger, the sampled incoherence parameter goes below the dotted lines for both models. 

\begin{figure}
  \begin{tabular}{cc}
      (a){ Viscous Burgers} & (b){ KdV} \\
      \includegraphics[width=0.5\textwidth]{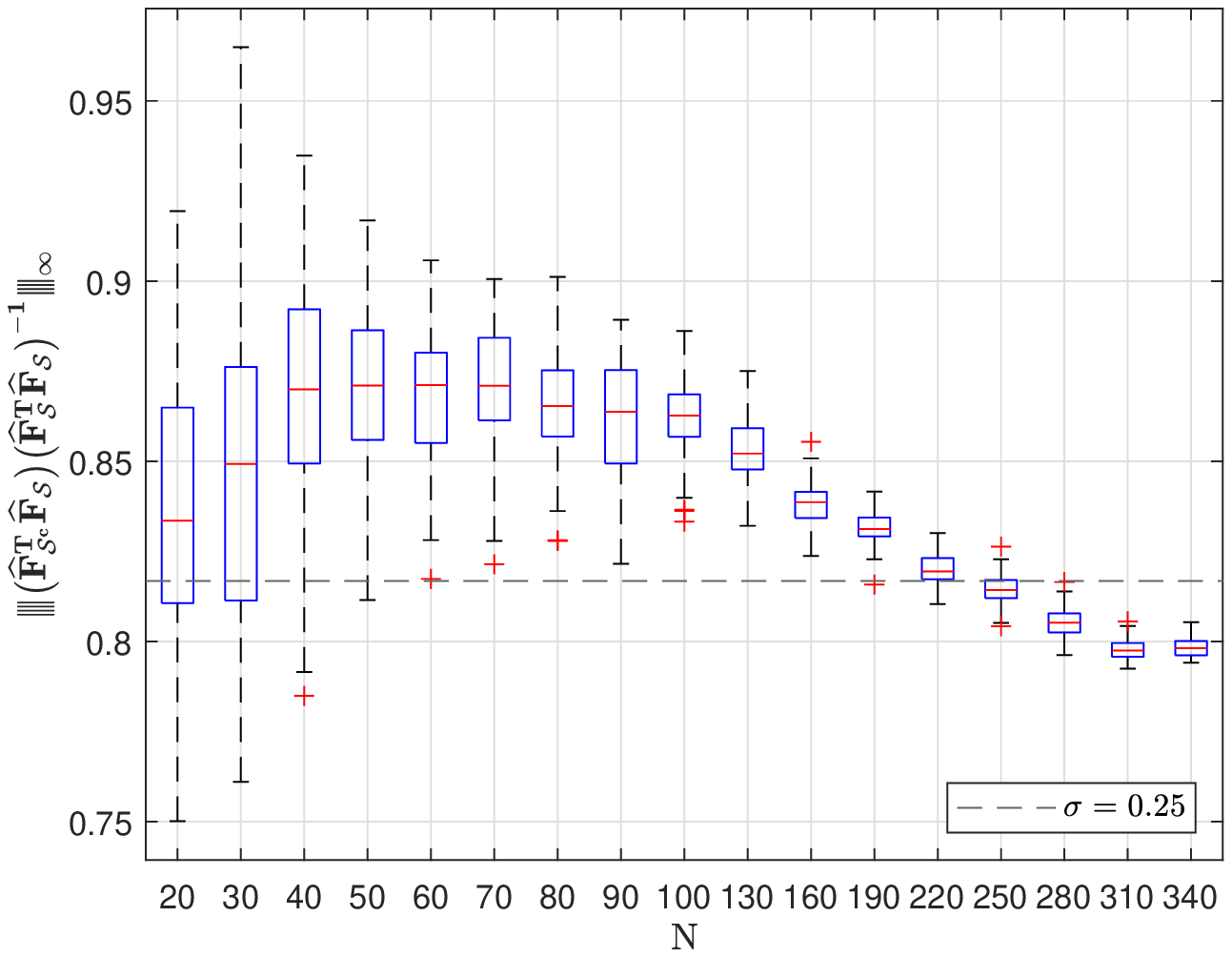}&
      \includegraphics[width=0.5\textwidth]{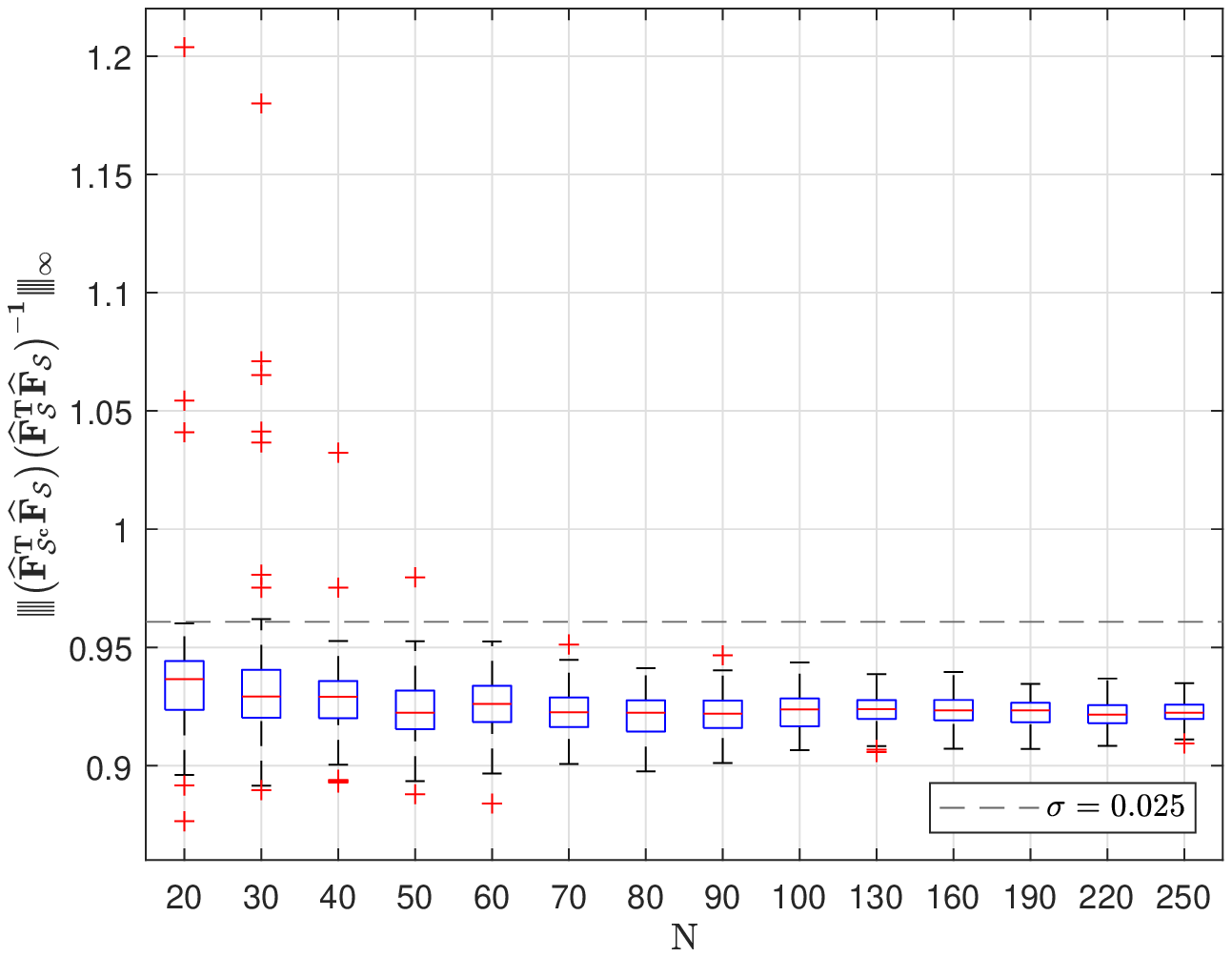}
  \end{tabular}
  \caption{Boxplots of $\VERT \widehat{\mathcal{Q}}_{N}\VERT_{\infty}$ versus $N$ are displayed for Viscous Burgers' equation in panel (a) and KdV equation in panel (b), respectively.}
  \label{fig_Inco}
\end{figure}

\subsection{Impact of $\bbeta^*_{\text{min}}$ in Signed-Support Recovery of $\ell_{1}$-PsLS} \label{ssec7.3}
Theorem \ref{thm1} states that as long as $\bbeta^*_{\min}:=\min_{i\in\mathcal{S}}|\bbeta^*_{i}|$ is beyond certain threshold, $\ell_{1}$-PsLS is signed-support recovery consistent.
In this subsection, we design an experiment to numerically confirm this claim. 
The experiment is performed over Viscous Burgers' equation by varying the coefficient $\nu$ in \eqref{eq_visBurgers} : we set $\nu=0.03,0.02,0.01,0.005$.
The Figure \ref{VB_beta_min} (a) displays the curves representing $\mathbb{P}[\mathbb{S}_{\pm}(\widehat{\bbeta})=\mathbb{S}_{\pm}(\bbeta^{*})]$ versus $N$ for each of the four cases.
Each point on each curve represents the average over $100$ trials. 
The Figure \ref{VB_beta_min} (b) exhibits the range of $\lambda_{N}$ for which $\ell_{1}$-PsLS finds the support of $\widehat{\bbeta}^{\lambda}$ that is contained within the true support, when $\nu$ is set as $0.005$.
More specifically, boxplots in (b) record the range of $\lambda_{N}$ that picks $\widehat{\bu}_{xx}$ as the selected argument.
In (a), we can check that, as the magnitude of $\min_{i\in\mathcal{S}}|\bbeta^*_{i}|$ decreases from $0.03$ to $0.01$, $\ell_{1}$-PsLS requires more data-points for the signed-support recovery,
and when $\min_{i\in\mathcal{S}}|\bbeta^*_{i}|$ drops to $0.005$, $\ell_{1}$-PsLS fails to recover the governing PDE.
On the other hand, (b) says that there exists a range of $\lambda_{N}$ for which $\ell_{1}$-PsLS can still recover a subset of $\bbeta^*$, while the perfect signed-support recovery is difficult.

\begin{figure} 
  \begin{tabular}{cc}
      (a) & (b) \\
      \includegraphics[width=0.5\textwidth]{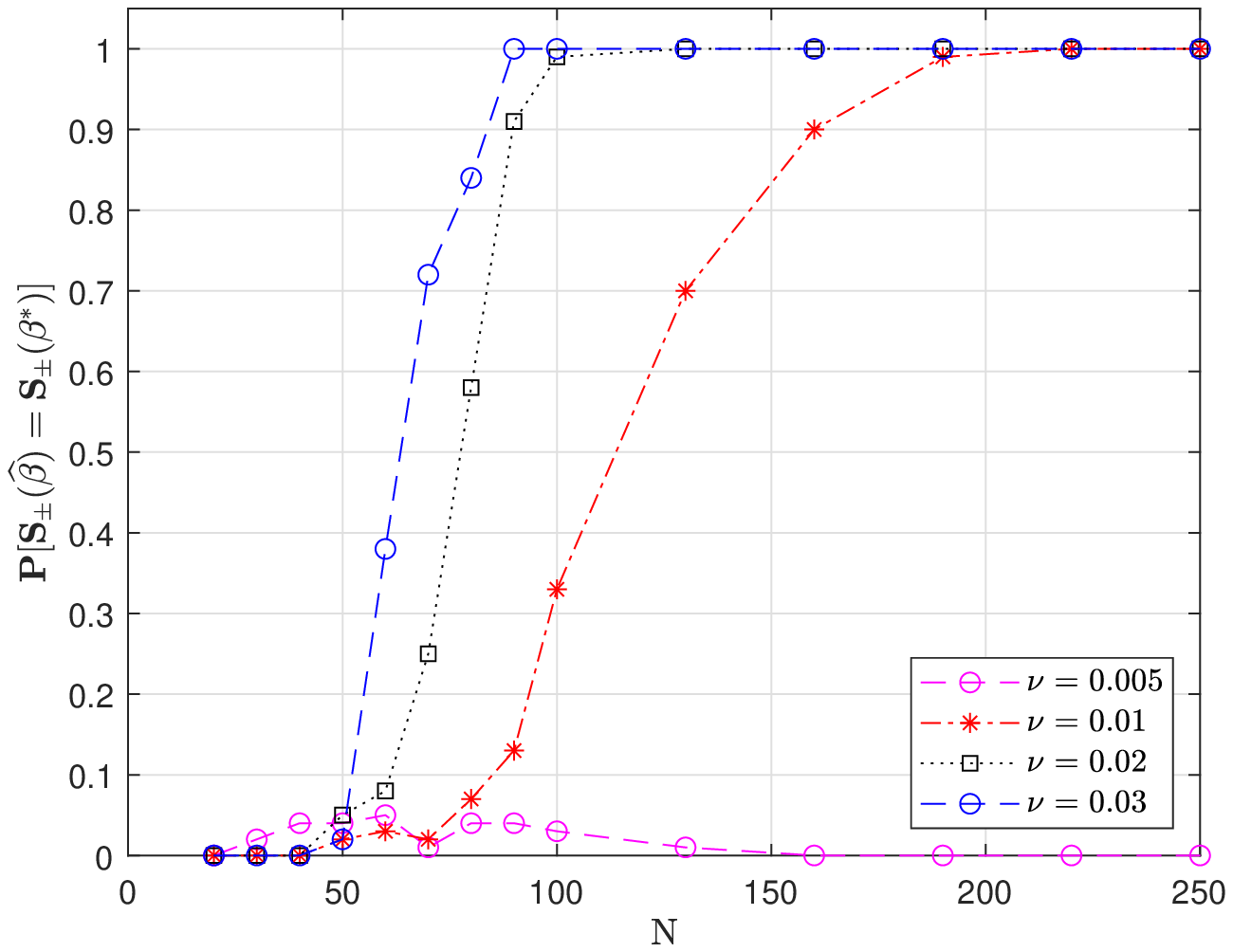}&
      \includegraphics[width=0.5\textwidth]{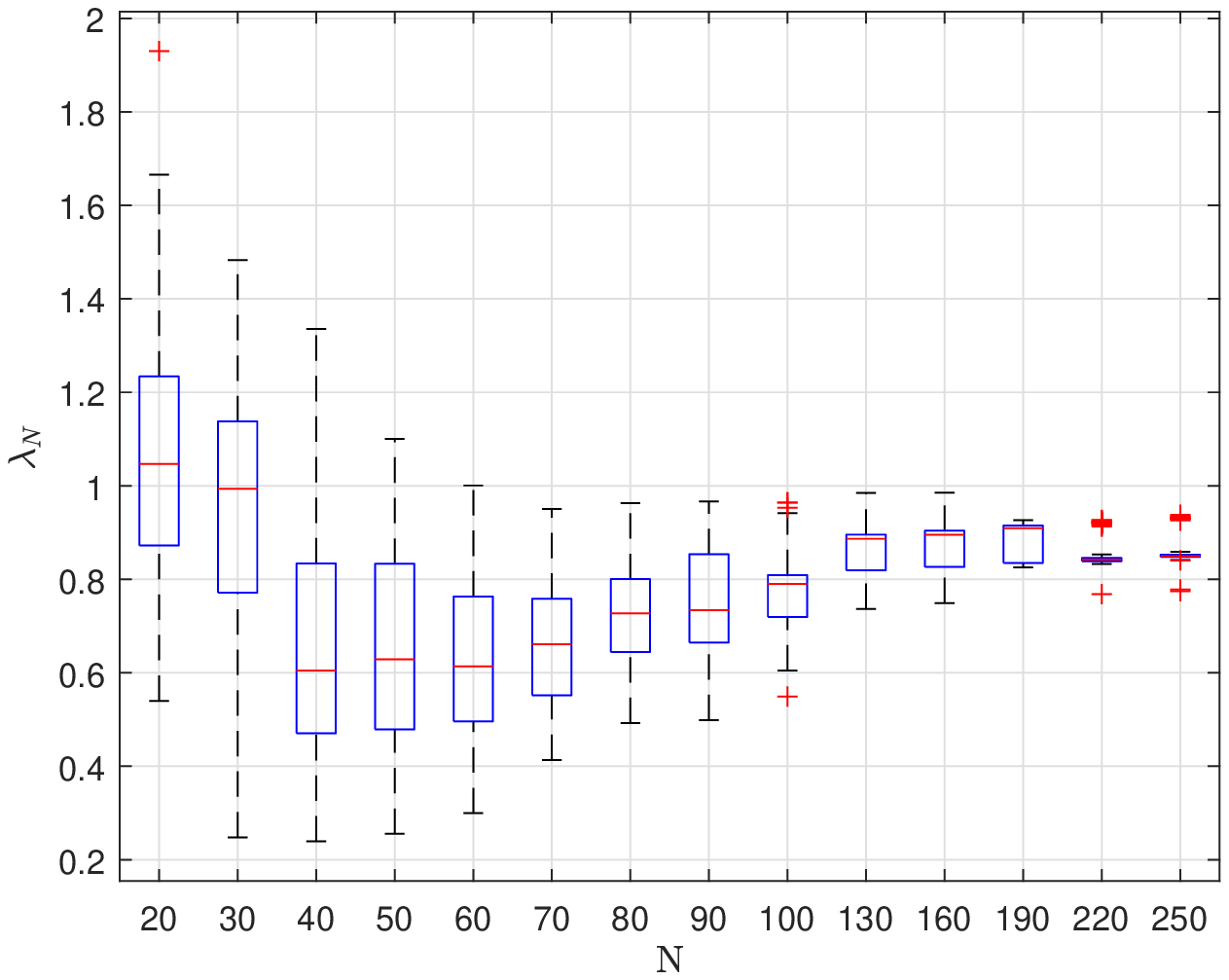}
  \end{tabular}
  \caption{Left panel (a) displays the curves representing 
   $\mathbb{P}[\mathbb{S}_{\pm}(\widehat{\bbeta})=\mathbb{S}_{\pm}(\bbeta^{*})]$ versus $N$, when $\nu=0.03,0.02,0.01,0.005$.
  Right panel (b) exhibits the range of $\lambda_{N}$ for which $\ell_{1}$-PsLS gives the solution $\widehat{\bbeta}^{\lambda}$ such that
  $\mathcal{S}(\widehat{\bbeta}^{\lambda})\subseteq\mathcal{S}(\bbeta^*)$ with respect to $N$, when $\nu$ is set as $0.005$. } 
  \label{VB_beta_min}
\end{figure}

\section{Conclusion}\label{sec_conclusion}
We provide a formal theoretical analysis on the PDE identification via $\ell_{1}$-regularized Pseudo Least Square method from the statistical point of view.
In this article, we assume that the differential equation governing the dynamic system can be represented as a linear combination of various linear and nonlinear differential terms.  
We employ the Local-Polynomial fitting and apply the $\ell_{1}$ penalty for model selection.  
A signed-support recovery of $\ell_{1}$-PsLS method with an exponential convergence rate is obtained under the classical mutual incoherence condition on the feature matrix $\bF$.
We divide the cases into two for the proof of the Theorem~\ref{thm1}.
Firstly, a signed-support recovery of $\ell_{1}$-PsLS method is shown with mutual incoherence assumption being imposed on the estimated feature matrix $\widehat{\bF}$. 
Then, we show $\widehat{\mathcal{Q}}_{N}$ gets close to $\mathcal{Q}^*$ under $\left\VERT\cdot\right\VERT_{\infty}$ ensuring the statement of the Theorem~\ref{thm1}.
We run numerical experiments on two popular PDE models, and the results from the experiments corroborate our theoretical predictions.
We present two directions to be explored based on the $\ell_{1}$-PsLS method proposed in our work.

\begin{enumerate}
    \item Recall that our theory utilizes the equivalent kernel theory for Local-Polynomial regression \cite{fan1997local}, stating that the higher-order Local-Polynomial smoothing is asymptotically equivalent to higher-order kernel smoothing.
    Due to this construction, our theory cannot characterize the convergence behavior of signed-support recovery of $\ell_{1}$-PsLS, when the number of observations is small.
    We conjecture that the uniform convergence rate of the Local-Polynomial estimator with exponential decay can be obtained in a non-asymptotic sense, by using a similar technique employed in ~\cite{audibert2007fast}.
    They impose an assumption that the regression function belongs to the H\"older class. They manipulate the closed-form solution of the Local-Polynomial estimator so that the difference of the estimator and the regression function has a special form that can be controlled by the Bernstein's inequality.
    It would be an interesting research direction to see whether this technique can be employed in our setting.
    
    \item The choice of the bandwidth parameter is essential in Local-Polynomial fitting, 
    thereby having a significant impact on support recovery of PDE problem via $\ell_{1}$-PsLS.
    It is worth noting that~\cite{liang2008parameter} employed the substitution method in~\cite{ruppert1995effective} based on the  asymptotic Mean Integrated Squared Error for the specific choices of the constant factors of the bandwidth parameter. 
    However, the method is only limited to the local-quadratic estimator and is not applicable to our setting, which requires a higher-order smoothing estimator.
    In our numerical experiments, we choose the constant factors of bandwidth parameters $h_{N}$ and $w_{M}$ manually.
    It only provides an ad-hoc guidance of bandwidth selection.
    Developing a data-driven bandwidth selection procedure for $\ell_{1}$-PsLS is a worthy topic for future research.
\end{enumerate}

\bibliographystyle{unsrt}
\bibliography{main}

\appendix
\section{Primal-Dual Witness construction} \label{app.pdw}
In this Section, we briefly rephrase the explanation of PDW construction in \cite{hastie2015statistical} for reader's convenience.
A primal-dual pair $(\widehat{\beta},\widehat{z})\in\mathbb{R}^{K \times K}$ is said to be optimal if $\widehat{\beta}$ is a minimizer of 
(\ref{eq.lassomat}) and $\widehat{z}\in\partial\|\widehat{\beta}\|_{1}$, where $\partial\|\widehat{\beta}\|_{1}$ denotes a sub-differential set of 
$\|\cdot\|_{1}$ evaluated at $\widehat{\beta}$.
Any such pair must satisfy zero-subgradient condition of (\ref{eq.lassomat}), which is as follows:
\begin{align}
    -\frac{1}{NM}\widehat{\bF}^T(\widehat{\bu}_t-\widehat{\bF}\widehat{\beta})+\lambda \widehat{z}=0\;,~\text{for}~\widehat{z}\in\partial\|\widehat{\beta}\|_1\;.\label{eqKKT_appndA}
\end{align}
Recall that we denote the ground-truth support of $\beta^{*}$ as $\mS$, and suppose that we know $\mS$ apriori.
For the ground-truth support set $\mS$ and its complement set $\mS^{c}$, PDW is said to be successful if the constructed tuple, 
$(\widehat{\beta}_{\mS},\widehat{\beta}_{\mS^c},\widehat{z}_{\mS},\widehat{z}_{\mS^c})$, is primal-dual optimal, and act as a witness for the fact that the LASSO finds the unique optimal solution with correct support set. 
We construct the tuple through the following three steps.
\begin{enumerate}
    \item Set $\widehat{\beta}_{\mS^{c}}=0$.
    \item Find $(\widehat{\beta}_{\mS},\widehat{z}_{\mS})$ by solving the $s$-dimensional oracle sub-problem
    \begin{equation*}
        \widehat{\beta}_{\mS}\in\argmin_{\beta_{\mS}\in\mathbb{R}^{s}}\bigg\{ \frac{1}{2NM}\left\|\widehat{\bu}_t-\widehat{\bF}_{\mS}\bbeta_{\mS}\right\|_{2}+\lambda\|\bbeta_{\mS}\|_1 \bigg\},
    \end{equation*}
    where $s$ is the cardinality of the set $\mS$.
    Thus $\widehat{z}_{\mS}\in\partial\|\widehat{\beta}_{\mS}\|_{1}$ satisfies the relation $-\frac{1}{NM}\widehat{\bF}_{\mS}^T(\widehat{\bu}_t-\widehat{\bF}_{\mS}\widehat{\beta}_{\mS})+\lambda \widehat{z}_{\mS}=0$.
    \item Solve for $\widehat{z}_{\mS^{c}}$ through the zero-subgradient equation (\ref{eqKKT_appndA}), and check whether or not the \emph{strict dual feasibility} condition $\left\| \widehat{z}_{\mS} \right\|_{\infty}<1$ holds.
\end{enumerate}

\section{Local-Polynomial estimator : Closed-form solutions} \label{app.loc}
Recall that we want to solve following two optimization problems for constructing $\widehat{\bu}_{t}$ and $\widehat{\bF}$, given the noisy observation $\mathcal{D}=\{\big(X_{i},t_{n},U_{i}^{n}\big)\mid ~i=0,\dots,M-1; n=0,\dots,N-1\}$.
\begin{align}
    &\bigg\{\widehat{b}_j(X_i,t)\bigg\}_{j=0,1,2}=\argmin_{b_j(t)\in\mathbb{R},0\leq j \leq 2 }\sum_{n=0}^{N-1}\bigg(U_i^n-\sum_{j=0}^{2}b_j(t)(t_n-t)^j\bigg)^2\mathcal{K}_{h_N}\bigg(t_n-t\bigg)\;,\nonumber\\
    &\quad\quad\text{for}~i=0,1,\dots,M-1\;;\label{eq_K_b_app}\\
    &\bigg\{\widehat{c}^p_j(x,t_n)\bigg\}_{j=0,1,\dots,p+1}=\argmin_{c_j(t)\in\mathbb{R},0\leq j \leq p+1 }\sum_{i=0}^{M-1}\bigg(U_i^n-\sum_{j=0}^{p+1}c^p_j(t)(X_i-x)^j\bigg)^2\mathcal{K}_{w_{M}}\bigg(X_i-x\bigg)\;\nonumber\\
    &\quad\quad\text{for}~n=0,1,\dots,N-1\; \text{and}~p=0,1,\dots,P_{\max}.\label{eq_K_c_app}
\end{align}
and set $\widehat{u}_t(X_i,t) = \widehat{b}_1(X_i,t)$ and $\widehat{\partial_x^pu}(x,t_n) = p!\widehat{c}^p_{p}(x,t_n)$.
Then, the standard weighted least-square theory leads to the solutions of \eqref{eq_K_b_app} and \eqref{eq_K_c_app}, respectively: 
\begin{align} 
    & \widehat{u}_t(X_i,t)=\xi^{T}_{1}\big( \mathbf{T_{1}}^{\text{T}} \mathbf{W_{t}}\mathbf{T_{1}} \big)^{-1}\mathbf{T_{1}}^{\text{T}}\mathbf{W_{t}}\mathbf{U_{i}}, \quad \forall i = 0,1,\dots,M-1 , \label{lpt_sol} \\
    & \widehat{\partial_x^pu}(x,t_n) = p!\xi^{T}_{p,x}\big( \mathbf{X_{p}}^{\text{T}} \mathbf{W_{x}}\mathbf{X_{p}} \big)^{-1}\mathbf{X_{p}}^{\text{T}}\mathbf{W_{x}}\mathbf{U^{n}}, \quad \forall p = 0,1,\dots,P_{\max}, \quad \forall n=0,1,\dots,N-1, \label{lpx_sol}
\end{align}
where $\mathbf{U_{i}}=[U_{i}^{0},\dots,U_{i}^{N-1}]^{\text{T}}$ and $\mathbf{U^{n}}=[U_{0}^{n},\dots,U_{M-1}^{n}]^{\text{T}}$, 
and
\begin{align*}
    \mathbf{T_{1}} := 
    \begin{bmatrix}
     	1 & t_{0}-t & \big( t_{0}-t \big)^{2}  \\
     	1 & t_{1}-t & \big( t_{1}-t \big)^{2}  \\
     	\vdots&\vdots& \vdots \\
     	1 & t_{N-1}-t & \big( t_{N-1}-t \big)^{2}
 	\end{bmatrix}\;,
 	\quad
	\mathbf{X_{p}} := 
    \begin{bmatrix}
     	1 & X_{0}-x & \cdots & \big( X_{0}-x \big)^{p+1}  \\
     	1 & X_{1}-x & \cdots & \big( X_{1}-x \big)^{p+1}  \\
     	\vdots&\vdots& \vdots & \vdots \\
     	1 & X_{M-1}-x & \cdots & \big( X_{M-1}-x \big)^{p+1}
 	\end{bmatrix}\;,
\end{align*}
for $p=0,\dots,P_{\max}$, and 
\begin{align*}
    &\mathbf{W_{t}}:=\text{diag}\big\{\mathcal{K}_{h_N}(t_0-t),\dots,\mathcal{K}_{h_N}(t_{N-1}-t)\big\}, \\
    &\mathbf{W_{x}}:=\text{diag}\big\{\mathcal{K}_{w_M}(X_0-x),\dots,\mathcal{K}_{w_M}(X_{M-1}-x)\big\},
\end{align*}
are $N\times N$ and $M\times M$ diagonal matrices of kernel weights,
and $\xi_{2}$ is the $3 \times 1$ vector having $1$ in the $2$nd entry and zeros in the other entries, and $\xi_{p,x}$ is the $(p+1)\times 1$ vector having $1$ in the $p$th entry and zeros in the other entries. 

\section{Proof of Proposition~\ref{prop1}} \label{app.proofprop1}
By the KKT-condition, any minimizer $\check{\bbeta}$ of~\eqref{eq.lassomat} satisfies:
\begin{align}
    -\frac{1}{NM}\widehat{\bF}^T(\widehat{\bu}_t-\widehat{\bF}\check{\bbeta})+\lambda_{N} \check{\mathbf{z}}=0\;,~\text{for}~\check{\mathbf{z}}\in\partial\|\check{\bbeta}\|_1\;.
\label{eqKKT}
\end{align}
Recall that $\Delta \bu_t=\widehat{\bu}_t-\bu_t$, $\Delta \bF=\widehat{\bF}-\bF$ denote the error terms.
By using the ground-truth PDE $u_{t}=\bF\beta^*$ and definitions of $\Delta{\bu_{t}}$ and $\Delta{\bF}$, we have $\widehat{\bu}_t=\widehat{\bF}\bbeta^*-\Delta \bF\bbeta^*+\Delta \bu_t$. 
Thus from~\eqref{eqKKT}, we get
\begin{align}
&\widehat{\bF}^T\widehat{\bF}(\check{\bbeta}-\bbeta^*)+\widehat{\bF}^T(\Delta \bF\bbeta^*-\Delta \bu_t)+\lambda_{N} NM\mathbf{z}=0\;~\label{eqKKT2}.
\end{align}
We decompose \eqref{eqKKT2} as follows:
\begin{align}
\begin{bmatrix}
\widehat{\bF}^T_S\widehat{\bF}_\mS&\widehat{\bF}^T_\mS\widehat{\bF}_{\mS^c}\\
\widehat{\bF}^T_{\mS^c}\widehat{\bF}_\mS&\widehat{\bF}^T_{\mS^c}\widehat{\bF}_{\mS^c}
\end{bmatrix}\begin{bmatrix}
\check{\bbeta}_{\mS}-\bbeta^*_{\mS}\\
0
\end{bmatrix}
+\begin{bmatrix}
\widehat{\bF}_\mS^T\\
\widehat{\bF}_{\mS^c}^T
\end{bmatrix}(\Delta \bF_\mS\bbeta_\mS^*-\Delta \bu_t)
+\lambda_{N} NM\begin{bmatrix}
\check{\mathbf{z}}_\mS\\
\check{\mathbf{z}}_{\mS^c}
\end{bmatrix}=\begin{bmatrix}
	\mathbf{0}\\\mathbf{0}
\end{bmatrix}\;,\label{matKKT}
\end{align}
where we used the fact $\bbeta^*_{\mS^c}=\mathbf{0}$ and
$\check{\bbeta}_{\mS^c}=0$ via PDW construction.
Solving \eqref{matKKT}, we have following two equalities:
\begin{align}
    &\widehat{\bF}^T_S\widehat{\bF}_\mS\big( \check{\bbeta}_{\mathcal{S}}-\bbeta^*_{\mathcal{S}} \big) + \widehat{\bF}_{\mS}^T(\Delta \bF_\mS\bbeta_\mS^*-\Delta \bu_t)+\lambda_{N}NM\check{\mathbf{z}}_\mS=0 \label{eq20}  \\  
    &\widehat{\bF}^T_{\mS^c}\widehat{\bF}_\mS\big( \check{\bbeta}_{\mathcal{S}}-\bbeta^*_{\mathcal{S}} \big) + \widehat{\bF}_{\mS^c}^T(\Delta \bF_\mS\bbeta_\mS^*-\Delta \bu_t)+\lambda_{N}NM\check{\mathbf{z}}_{\mS^c}=0  \label{eq21}        
\end{align}
Using the minimum eigen-value condition in the assumption \eqref{assum3}, from \eqref{eq20}, we have
\begin{equation} \label{eq_beta_min}
    \check{\bbeta}_{\mathcal{S}}-\bbeta^*_{\mathcal{S}}
    = \big( \widehat{\bF}^T_S\widehat{\bF}_\mS \big)^{-1}
    \bigg( \widehat{\bF}_{\mS}^T(\Delta \bu_t - \Delta \bF_\mS\bbeta_\mS^*) - \lambda_{N}NM\check{\mathbf{z}}_\mS \bigg).
\end{equation}
Plugging \eqref{eq_beta_min} into \eqref{eq21} gives:
\begin{align*}
\check{\mathbf{z}}_{\mS^c}=\widehat{\bF}_{\mS^c}^T\widehat{\bF}_\mS(\widehat{\bF}_{\mS}^T\widehat{\bF}_\mS)^{-1}\mathbf{z}_\mS+\frac{1}{\lambda MN}\widehat{\bF}_{\mS^c}^T{\bf{\Pi}_{\mS^{\perp}}}(\Delta \bu_t-\Delta \bF_\mS\bbeta_\mS^*)\;,
\end{align*}
where ${\bf{\Pi}_{\mS^{\perp}}}={\bf{I}}-\widehat{\bF}_\mS(\widehat{\bF}_\mS^T\widehat{\bF}_S)^{-1}\widehat{\bF}_\mS^T$ is an orthogonal projection operator on the column space of $\widehat{\bF}_\mS$.
By the complementary slackness condition, for $j\in\mathcal{S}^{c}$, $|\check{\mathbf{z}}_j|< 1$ implies $\check{\bbeta}_j=\mathbf{0}$, which guarantees the proper support recovery. i.e., $\mathcal{S}(\check{\bbeta})\subseteq\mathcal{S}(\bbeta^*)$.
Now, we can focus on proving that, as $N,M\to \infty$, for $\mu$ in \eqref{assum3}, 
$\mathbb{P}\big[\max_{j\in \mS^c}|\widetilde{Z}_j|\geq \mu \big]\to 0$, for $\widetilde{Z}_j = [\widehat{\bF}_{\mS^c}]^T_j{\bf\Pi_{\mS^{\perp}}}\frac{\Delta \bu_t-\Delta \bF_\mS\bbeta_\mS^*}{\lambda NM}$, $[\widehat{\bF}_{\mS^c}]_j$ is the $j$-th column of $\widehat{\bF}_{\mS^c}$. 
By the following lemma, we claim that to prove (\rn{1}) of Proposition~\ref{prop1}, it suffices to bound $\ell_\infty$-norm of the PDE estimation error $\btau$.

\begin{lemma} \label{Zlemma}
For any $\varepsilon>0$:
\begin{align*}
\mathbb{P}\bigg[\max_{j\in \mS^c}\left|\widetilde{Z}_j\right|\geq\varepsilon\bigg]\leq  \mathbb{P}\bigg[\left\|\btau\right\|_\infty\geq\frac{\lambda\varepsilon}{\sqrt{K}}\bigg]\;.
\end{align*}
\end{lemma}

\begin{proof}
	\begin{align*}
	\mathbb{P}\Bigg[\left\|\widehat{\bF}_{\mS^{c}}^T{\bf\Pi_{\mS^{\perp}}}\frac{\btau}{\lambda NM}\right\|_{\infty}\geq\varepsilon\Bigg]
	&\leq \mathbb{P}\Bigg[\left\|\widehat{\bF}^{T}{\bf\Pi_{\mS^{\perp}}}\frac{\btau}{\lambda NM}\right\|_{2}\geq\varepsilon\Bigg] \\
	&\leq \mathbb{P}\Bigg[\left\VERT\bf{\Pi_{\mS^{\perp}}}\big(\widehat{\bF}\big)\right\VERT_{2}\left\|\frac{\btau}{\lambda NM}\right\|_2\geq\varepsilon\Bigg] \\
	&\leq \mathbb{P}\Bigg[\left\VERT\widehat{\bF}\right\VERT_{F}\,\left\|\frac{\btau}{\lambda NM}\right\|_2\geq\varepsilon\Bigg] \\
	&\leq \mathbb{P}\Bigg[\|\btau\|_2\geq\lambda \varepsilon\sqrt{\frac{NM}{K}}\Bigg] \\
	&\leq \mathbb{P}\Bigg[\|\btau\|_\infty\geq\frac{\lambda\varepsilon}{\sqrt{K}}\Bigg].
	\end{align*}
In the second inequality, we use the definition of spectral norm of matrix, and in the third inequality, we use the fact $\VERT\bf\Pi_{\mS^{\perp}}\VERT_{2}=1$.
In the fourth inequality, the condition $\frac{1}{\sqrt{NM}} \max_{j=1,\dots,K} \| \widehat{\bF}_{j} \|_{2} \leq 1$ is used, giving us $\VERT\widehat{\bF}\VERT_{F}\leq\sqrt{KNM}$.
In the last inequality, we use $\|\btau\|_{2}\leq\sqrt{NM}\|\btau\|_{\infty}$.
\end{proof}

\subsection{Sufficient conditions for bounding $\widehat{\bu}_t-\bu_t$} \label{app.B1}
\begin{lemma}\label{Lemm.Dt}
Let $\mathcal{K}^*_{\max}=\|\mathcal{K}^*\|_\infty$,  $B_N$ be an arbitrary increasing sequence $B_N\to\infty$ as $N\to\infty$, and  $B^{'}_N=B_N+\|u\|_{L^\infty(\Omega)}$. For any $i=0,1,\dots,M$ and arbitrary real $r$, there exist finite positive constants $A(X_i),C^*(X_i),a_0,b_0,c_0,$ and $d_0(X_i)$ which do not depend on the temporal sample size $N$, such that for any $\alpha>1$ and 
\begin{align*}
&\varepsilon_N^*(X_i,r,\alpha)>\\
&\max\bigg\{3|C^*(X_i)|h_N^{2},\frac{6\mathcal{K}^*_{\max}B^{'}_N}{Nh_N^2},6\frac{A(X_i)\big(B^{'}_N\big)^{-1}}{h_{N}},\frac{6B^{'}_N\mathcal{K}_{\max}^*(a_0\ln N+r)\ln N}{h_N^2N},12\sqrt{\alpha}d_0(X_i)\sqrt{\frac{\ln1/h_N}{h_N^{3}N}}\bigg\}\;,	
\end{align*}
as long as $N$ is sufficiently large,  we have:
\begin{align*}
\mathbb{P}\Big[\sup_{t\in[0,T]}|\Delta u_t(X_i,t)|>\varepsilon_N^*(X_i,r,\alpha)\Big]<2N \exp\bigg(-\frac{B_N^2}{2\sigma^2}\bigg)+b_0\exp\big(-c_0r\big)+4\sqrt{2}\eta^4h_N^\alpha\;.
\end{align*}
\end{lemma}

\begin{proof} 
In the following argument, we fix some $i=0,\cdots, M-1$ and omit the dependence  on $X_i$ in the notations. Let $B^{'}_N=B_N+\|u\|_{L^\infty(\Omega)}$ with $B_N$ being a sequence of increasing positive numbers such that $B_N\to \infty$ as $N\to\infty$, then define the truncated estimate
\begin{align}
	\widehat{u_t}^{B^{'}_N}(X_i,t)&=\frac{1}{Nh^2_N}\sum_{n=0}^{N-1}\mathcal{K}^*\bigg(\frac{t_n-t}{h_N}\bigg)U_i^n I\{|U_i^n|<B^{'}_N\} \label{eq_uthat} \\
	&=\frac{1}{h^2_N}\iint_{|y|<B^{'}_N}\mathcal{K}^*\bigg(\frac{z-t}{h_N}\bigg)y\,df_N(z,y)\nonumber,
\end{align}
where $f_N(\cdot,\cdot):=f_N(\cdot,\cdot|X_i)$ is the empirical distribution of $(t_n,U_i^n)$ conditioned on the space $X_i$. For any $(X_i,t)$, decomposing the estimation error of the temporal partial derivative as follows
\begin{align*}
    \widehat{u}_{t}-u_{t} 
    = \underbrace{\bigg( \widehat{u}_{t} - \widehat{u_t}^{B_N'} - \mathbb{E} \big( \widehat{u}_{t} - \widehat{u_t}^{B^{'}_N} \big) \bigg)}_{\text{ Asymptotic deviation on the truncation error}}
    + \underbrace{\bigg( \widehat{u_t}^{B^{'}_N} - \mathbb{E} \widehat{u_t}^{B^{'}_N} \bigg)}_{\substack{\text{Asymptotic deviation of} \\ \text{truncated estimator}}}
    + \underbrace{\bigg( \mathbb{E}\widehat{u}_{t} - u_{t} \bigg)}_{\text{Asymptotic bias}},
\end{align*}
we will prove that the error is bounded (in probability) by showing each component is bounded.\\[3pt]

\noindent\textbf{\textit{Component 1. Asymptotic deviation on the truncation error}}: Notice that for any $\varepsilon_{0,N}\geq \frac{\mathcal{K}^*_{\max}B^{'}_N}{Nh_N^2}$:
\begin{align*}
\mathbb{P}&\Big[\sup_t|\widehat{u_t}-\widehat{u_t}^{B^{'}_N}|>\varepsilon_{0,N}\Big]
=\mathbb{P}\Big[\sup_t \left|\frac{1}{Nh^2_N}\sum_{n=0}^{N-1}\mathcal{K}^*\bigg(\frac{t_n-t}{h_N}\bigg)U_i^n I\{|U_i^n|\geq B^{'}_N\}\right|>\varepsilon_{0,N}\Big]\\
&\leq \mathbb{P}\Big[\frac{\mathcal{K}^*_{\max}}{Nh_N^2}\sum_{n=0}^{N-1}|U_i^n|I\{|U_i^n|\geq B^{'}_N\}>\varepsilon_{0,N}\Big]
\leq \mathbb{P}\Big[\exists n=0,1,\cdots, N-1,~|U_i^n|\geq B{'}_N\Big]\\
& = \mathbb{P}\Big[\max_{n=0,1,\cdots,N-1}|U_i^n|\geq B{'}_N\Big]
\leq \mathbb{P}\Big[\max_{n=0,1,\cdots, N-1}|U_i^n-u_i^n|\geq B_N\Big]\leq 2N \exp\bigg(-\frac{B_N^2}{2\sigma^2}\bigg)\;
\end{align*}
where $\sigma$ denotes the standard deviation of the Gaussian noise added on the data. On the other hand, from Proposition 1 of~\cite{mack1982weak}:
\begin{align*}
\E|\widehat{u_t}-\widehat{u_t}^{B^{'}_N}|\leq \frac{A\big(B^{'}_N\big)^{-1}}{h_{N}}\;.
\end{align*}
for $A=\int|\mathcal{K}(\zeta)|\,d\zeta\times\sup_{t}\int|y|f(t,y|X_i)\,dy$ with $f(\cdot,\cdot|X_i)$ as the distribution of $(t,U(X_i,t))$; hence for any $\varepsilon_{1,N}\geq 2\max\{\frac{\mathcal{K}^*_{\max}B_N}{Nh_N^2},\frac{A\big(B^{'}_N\big)^{-1}}{h_{N}}\}$, we have:
\begin{align*}
&\mathbb{P}\Big[\sup_t|\widehat{u_t}(X_i,t)-\widehat{u_t}^{B^{'}_N}(X_i,t)-(\E(\widehat{u_t}(X_i,t)-\widehat{u_t}^{B^{'}_N}(X_i,t)))|>\varepsilon_{1,N}\Big]\leq 2N \exp\bigg(-\frac{B_N^2}{2\sigma^2}\bigg)\;.
\end{align*}

\noindent\textbf{\textit{Component 2. Asymptotic deviation of truncated estimator}}: Observe that
\begin{align}
    \widehat{u_t}^{B^{'}_N} - \mathbb{E}\big( \widehat{u_t}^{B^{'}_N} \big)  
    &= \frac{1}{\sqrt{N}h_N^2}\int_{z\in\mathbb{R}}\int_{|y|\leq B^{'}_N} \mathcal{K}^*\bigg( \frac{z-t}{h_N} \bigg) y 
    d_z d_y\underbrace{\left(\sqrt{N} (f_N(z,y)-f(z,y))\right)}_{:=Z_N(z,y)} \nonumber \\
    &= \frac{1}{\sqrt{N}h^2_N}\int_{z\in\mathbb{R}} \mathcal{K}^*\bigg( \frac{z-t}{h_N} \bigg) d_z U^{B^{'}_N}(z), \label{eq5}
\end{align}
where $U^{B^{'}_N}(z)$ is defined by
\begin{align*}
    U^{B'_{N}}(z):=\int_{|y|\leq B^{'}_N} y d_y Z_N(z,y).
\end{align*}
Let $\mathcal{T}:\mathbb{R}^{2}\rightarrow{[0,1]^2}$ be the Rosenblatt transformation~\cite{rosenblatt1952remarks},
and define $\mathcal{B}$ as the 2-dimensional solution path of the Brownian Bridge which takes the transformed $\mathcal{T}(z,y)$ as an argument; then we have
\begin{align} \label{eq7}
    U^{B'_{N}}(z):=\int_{|y|\leq B^{'}_N} y d_y \big\{ Z_N(z,y)-\mathcal{B}(\mathcal{T}(z,y)) \big\} + \int_{|y|\leq B^{'}_N} y d_y \mathcal{B}(\mathcal{T}(z,y)).
\end{align}
Plug in \eqref{eq7} to \eqref{eq5}, we  get
\begin{align*}
     \widehat{u_t}^{B^{'}_N} - \mathbb{E}\big( \widehat{u_t}^{B^{'}_N} \big)   
    &= \underbrace{\frac{1}{\sqrt{N}h^2_N}\int_{z\in\mathbb{R}} \mathcal{K}^*\bigg( \frac{z-t}{h_N} \bigg) d_z \int_{|y|\leq B^{'}_N} y d_y \big\{ Z_N(z,y)-\mathcal{B}(\mathcal{T}(z,y)) \big\}}_{\gamma_N(t)} \\
    &+ \frac{1}{\sqrt{N}}\underbrace{\frac{1}{h^2_N}\int_{z\in\mathbb{R}} \int_{|y|\leq B^{'}_N} \mathcal{K}^*\bigg( \frac{z-t}{h_N} \bigg) y d_z d_y \mathcal{B}(\mathcal{T}(z,y))}_{\rho_N(t)}\\
    &= \gamma_N(t) + \frac{1}{\sqrt{N}}\rho_N(t).
\end{align*}
In the following, we bound $\gamma_N$ and $\rho_N(t)/\sqrt{N}$ respectively.
\begin{enumerate}
    \item \textit{Bound for $\gamma_N(t)$:} Since  $\mathcal{K}^*$ has compact support, applying integration by parts on $\gamma_{N}(t)$ gives
\begin{align}
    \gamma_{N}(t)
    &=-\frac{1}{\sqrt{N}h^2_N}\int_{z\in\mathbb{R}} \int_{|y|\leq B^{'}_N} y d_y \big\{ Z_N(z,y)-\mathcal{B}(\mathcal{T}(z,y)) \big\} d_{z} \mathcal{K}^*\bigg( \frac{z-t}{h_N} \bigg) \nonumber \\ 
    &\leq \frac{2B_N^{'}\mathcal{K}^{*}_{\max}}{\sqrt{N}h^2_N}\sup\limits_{z,y}\bigg\lvert Z_N(z,y)-\mathcal{B}(\mathcal{T}(z,y)) \bigg\rvert.\label{bound_gamma}
\end{align}
By Tusnady's strong approximation result~\cite{tusnady1977remark}, there exist absolute positive constants $a_0,b_0$ and $c_0$ such that
\begin{align}
    \mathbb{P}\bigg[ \sup_{z,y} \bigg\lvert Z_N(z,y)-\mathcal{B}(\mathcal{T}(z,y)) \bigg\rvert > \frac{\big(a_0\ln N + r\big)\ln N} {\sqrt{N}}\bigg] < b_0 \exp(-c_0r)\label{ineq_boundzb}
\end{align}
holds for any real $r$. Therefore, if we take $\varepsilon'_{2,N}(r)=\frac{2B_N^{'}\mathcal{K}^{*}_{\max}(a_0\ln N + r)\ln N}{N h^2_N}$, combining~\eqref{bound_gamma} and~\eqref{ineq_boundzb} gives
\begin{align}
    \mathbb{P}\bigg[ \sup_{t} \lvert \gamma_N(t) \rvert > \varepsilon'_{2,N}(r) \bigg]
    < b_0 \exp(-c_0r).\label{eq_gamma}
\end{align}
\item \textit{Bound for $\rho_N(t)/\sqrt{N}$:} Similarly to (7) of~\cite{mack1982weak}, we have
\begin{align*}
\frac{h_N^{3/2}\sup_t|\rho_N(t)|}{\sqrt{\ln\frac{1}{h_N}}}&\leq 
\underbrace{16(\ln V)^{1/2}S^{1/2}\Big(\ln\frac{1}{h_N}\Big)^{-1/2}\int|\zeta|^{1/2}\,|d\mathcal{K}^*(\zeta)|}_{:=Q_{1,N}}\\
&+\underbrace{16\sqrt{2}h_N^{-1/2}\Big(\ln\frac{1}{h_N}\Big)^{-1/2}\int q(Sh_N|\zeta|)\,|d\mathcal{K}^*(\zeta)|}_{:=Q_{2,N}},
\end{align*}
where  $V$ is a random variable satisfying $\E V\leq 4\sqrt{2}\eta^4$ (recall that $\eta^2 := \max\limits_{i,n}\E(U_i^n)^2$), $q(r):=\int_0^r\frac{1}{2}(\frac{1}{y}\ln\frac{1}{y})^{1/2}\,dy$, $S:=\sup_z\int y^2f(z,y)\,dy$. Let $d_0=16\sqrt{2}S^{1/2}\int|\zeta|^{1/2}|d\mathcal{K}^*(\zeta)|$, which is a positive number independent of either $N$ or $M$. Consider the following inequality for an arbitrary $\varepsilon$
\begin{align}
    \mathbb{P}\Bigg( \frac{h_N^{{ 3}/2}\sup_t|\rho_N(t)|}{\sqrt{\ln\frac{1}{h_N}}}\geq\varepsilon \Bigg) 
    &\leq \mathbb{P}\bigg( Q_{1,N} \geq \frac{\varepsilon}{2} \bigg) + \mathbb{P}\bigg( Q_{2,N} \geq \frac{\varepsilon}{2} \bigg) \nonumber \\
    &\leq \mathbb{P}\Bigg( \big(\ln V\big)^{1/2} \geq \frac{\varepsilon\big(\ln\frac{1}{h_{N}}\big)^{1/2}}{2d_0} \Bigg) + \mathbb{P}\bigg( Q_{2,N} \geq \frac{\varepsilon}{2} \bigg) \nonumber \\
    &\leq 4\sqrt{2}\eta^{4}\exp\bigg( -\frac{\varepsilon^{2} \big( \ln\frac{1}{h_{N}} \big) }{4d_0^{2}} \bigg) + \mathbb{P}\bigg( Q_{2,N} \geq \frac{\varepsilon}{2} \bigg), \label{eq_28}
\end{align}
where the Markov Inequality is used in the last inequality. 
Setting $\varepsilon''_{2,N}=\varepsilon\sqrt{\frac{\ln\frac{1}{h_{N}}}{Nh_{N}^{{ 3}}}}$ gives
\begin{align*}
    \mathbb{P}\Bigg( \frac{\sup_t|\rho_N(t)|}{\sqrt{N}} \geq \varepsilon''_{2,N} \Bigg) 
    \leq 4\sqrt{2}\eta^{4}\exp\bigg( -\frac{\varepsilon^{2} \big( \ln\frac{1}{h_{N}} \big) }{4d_0^{2}} \bigg) + 
    \mathbb{P}\bigg( Q_{2,N} \geq \frac{\varepsilon}{2} \bigg).
\end{align*}
Notice that $Q_{2,N}$ converges to $d_0$ by Silverman~\cite{silverman1978weak}.  For any arbitrary $\alpha>1$, if $\varepsilon=2\sqrt{\alpha} d_0$, there exists a positive integer $N(\alpha)$ such that as long as $N>N(\alpha)$, we have $Q_{2,N}<\sqrt{\alpha}d_0$; hence the second probability in~\eqref{eq_28} becomes $0$. Considering that $\varepsilon_{2,N}''$ now depends on $\alpha$, we write it as $\varepsilon_{2,N}''(\alpha)$, and for sufficiently large $N$ ($N>N(\alpha)$), we obtain
\begin{align}
    \mathbb{P}\Bigg( \frac{\sup_t|\rho_N(t)|}{\sqrt{N}} \geq \varepsilon''_{2,N}(\alpha) \Bigg)
    \leq 4\sqrt{2}\eta^4h^{\alpha}_N\;.\label{eq_rho}
\end{align}
\end{enumerate}

Now if we take $\varepsilon_{2,N}(r,\alpha)=2\max\{\varepsilon'_{2,N}(r),\varepsilon''_{2,N}(\alpha)\}$ and combine~\eqref{eq_gamma} with~\eqref{eq_rho}, we have
\begin{align*}
\mathbb{P}\left(\sup_t|\widehat{u_t}^{B^{'}_N} - \mathbb{E}\big( \widehat{u_t}^{B^{'}_N} \big)|>\varepsilon_{2,N}(r,\alpha)\right)<b_0 \exp(-c_0r)+4\sqrt{2}\eta^{4}h_N^\alpha 
\end{align*}
\noindent\textbf{\textit{Component 3. Asymptotic  bias}}: From~\cite{fan1997local}, the asymptotic bias of the estimator directly follows
\begin{align*}
	\E\big(\widehat{u_t}\big) - u_t=C^*h_N^{2}\;.
\end{align*}
for some constant $C^*$ independent of $N$. 
Specifically, since we fit a degree $2$ polynomial to obtain $\widehat{u}_t(X_i,\cdot)$, we plug $p=2$ and $\nu=1$ in the expression of asymptotic bias of the estimator. See page $83$ of the paper \cite{fan1997local} for the expression. 
Taking $\varepsilon_{3,N}=|C^*|h_{N}^{2}$, we have $\mathbb{P}\left(|\E \big( \widehat{u_t} \big) - u_t|>\varepsilon_{3,N}\right)=0$.\\[5pt]
Combining all the three components above and taking  $\varepsilon_N^*(r,\alpha)>3\max\{\varepsilon_{1,N},\varepsilon_{2,N}(r,\alpha),\varepsilon_{3,N}\}$ gives the desired result.
\end{proof}

\subsection{Sufficient conditions for bounding $(\widehat{\bF}-\bF)\bbeta^*$}

For the $p$-th order partial derivative estimators with respect to $x$, we have results similarly to Lemma~\ref{Lemm.Dt}.
\begin{lemma}\label{Lemm.Dx}
Fix an order $p\geq 0$, and let $B_M$ be an arbitrary increasing sequence $B_M\to\infty$ as $M\to\infty$, and  $B^{'}_M=B_M+\|u\|_{L^\infty(\Omega)}$. For any $n=0,1,\dots,N-1$ and arbitrary $r$, there exist finite positive constants $A_p(t_n),C^*(t_n),a_0,b_0,c_0,$ and $d_0(t_n)$ which do not depend on the spacial sample size $M$, such that for any $\alpha>1$ and 
\begin{align*}
&\varepsilon_{M,p}^*(t_n,r,\alpha)>\\
&\max\bigg\{3|C^*(t_n)|w_{M}^{{ 2}},\frac{6p!\mathcal{K}^*_{\max}B^{'}_M}{Mw_{M}^{1+p}},6\frac{p!A_p(t_n)(B'_M)^{-1}}{w_{M}^{p}},\frac{6p!B^{'}_M(a_0\ln M+r)\ln M}{w_{M}^{1+p}M},12p!\sqrt{\alpha}d_0(t_n)\sqrt{\frac{\ln1/w_{M}}{w_{M}^{2p+{ 1}}M}}\bigg\}\;,	
\end{align*}
as long as $M>M(\alpha)$ for some positive integer $M(\alpha)$,  we have:
\begin{align*}
\mathbb{P}\Big[\sup_{x\in[0, X_{\max})}|\widehat{\partial_x^pu}(x,t_n)-\partial_x^pu(x,t_n)|>\varepsilon_{M,p}^*\Big]<2M \exp\bigg(-\frac{B_M^2}{2\sigma^2}\bigg)+b_0\exp(-c_0r)+4\sqrt{2}\eta^4w_M^\alpha\;.
\end{align*}
\end{lemma}

\begin{proof}
Notice that for any fixed temporal point $t_n$, $n=0,1,\dots, N-1$, the estimation for the $p$-th order partial derivative takes the form
\begin{align}
    \widehat{\partial_x^pu}(x,t_n) = \frac{p!}{Mw_M^{p+1}}\sum_{i=1}^M\mathcal{K}^*\left(\frac{X_i-x}{w_M}\right)U_i^n
\end{align}
with probability $1$~\cite{fan2018local}. Hence, we can prove the desired result by substituting $h_N^2$ with  $w_M^{p+1}/p!$ in~\eqref{eq_uthat} and follow the proof of Lemma~\ref{Lemm.Dt} and keeping in mind that the constants now depend on $t_n$ and not on $M$. Notice that the kernel $\mathcal{K}$ used for the spacial dimension may be different from that used for the temporal; this can be addressed by taking $\mathcal{K}^*_{\max}$ to be the larger value between their $\ell_\infty$-norms. Finally, given any fixed $t_n$, the asymptotic bias takes the form
$$\mathbb{E} \big( \widehat{\partial_x^pu} \big) -\partial_x^pu=C_p^{*}w_M^{2}$$
where ${ C_p^{*}\leq\max_{p=0,1,\dots,P_{\max}}\big\{\int z^{p+{1}}\mathcal{K}^*_{p}(z)\,dz\big\}\,\frac{p!}{(p+2)!}\partial_x^{p+{1}}u:=C^*}$ for any $0\leq p\leq P_{\max}$.
Here, since we fit the Local-Polynomial with degree $\ell+{1}$ to obtain $\widehat{\partial_x^\ell u}$, we plug $p=\ell+{1}$ and $\nu=\ell$ in the expression of asymptotic bias in \cite{fan1997local}.
\end{proof}

\noindent As for the product terms:
\begin{lemma}\label{Lemm.Dxx}
Fix any two orders $p,q\geq 0$, and let $B_M$ be an arbitrary increasing sequence $B_M\to\infty$ as $M\to\infty$, and  $B^{'}_M=B_M+\|u\|_{L^\infty(\Omega)}$. For any $n=0,1,\dots,N-1$ and arbitrary $r$, there exist finite positive constants $A(t_n),C^*(t_n),a_0,b_0,c_0,$ and $d_0(t_n)$ which do not depend on the spacial sample size $M$, such that for any $\alpha>1$ and 
\begin{align*} 
\varepsilon_{M,p,q}^{**}>\max\{3\|\partial_x^pu(\cdot,t_n)\|_\infty\varepsilon^*_{M,p},~3\|\partial_x^qu(\cdot,t_n)\|_\infty\varepsilon^*_{M,q},~3(\varepsilon^*_{M,p})^2,~3(\varepsilon^*_{M,q})^2\}
\end{align*}
as long as $M>M(\alpha)$ for some positive integer $M(\alpha)$,  we have:
\begin{align*}
&\frac{1}{4}\mathbb{P}\Big[\sup_{x\in[0, X_{\max})}|\widehat{\partial_x^pu}(x,t_n)\widehat{\partial_x^qu}(x,t_n)-\partial_x^pu(x,t_n)\partial_x^qu(x,t_n)|>\varepsilon_{M,p,q}^{**}\Big] \nonumber\\
&< 2M \exp(-\frac{B_M^2}{2\sigma^2})+b_0\exp(-c_0r)+4\sqrt{2}\eta^4w_M^{\alpha}\;,
\end{align*}	
Here $\varepsilon^*_{M,p}$ and $\varepsilon^*_{M,q}$ (depending on $B_M'$) are the thresholds in Lemma~\ref{Lemm.Dx} for the sup-norm bound of the estimator $\widehat{\partial_x^pu}$ and $\widehat{\partial_x^qu}$, respectively,
\end{lemma}

\begin{proof}
Notice that for any $\varepsilon>0$, we can bound the probability:
\begin{align*}
&\mathbb{P}\Big[\sup_{x\in[0, X_{\max})}|\widehat{\partial_x^pu}(x,t_n)\widehat{\partial_x^qu}(x,t_n)-\partial_x^pu(x,t_n)\partial_x^qu(x,t_n)|>\varepsilon\Big] \\
&\leq\mathbb{P}\Big[\|\partial_x^pu(\cdot,t_n)\|_{\infty}\sup_{x\in[0, X_{\max})}|\Delta\partial_x^qu(x,t_n)|>\varepsilon/3\Big] \\
&+ \mathbb{P}\Big[\|\partial_x^qu(\cdot,t_n)\|_{\infty}\sup_{x\in[0, X_{\max})}|\Delta\partial_x^pu(x,t_n)|>\varepsilon/3\Big]\\
&+ \mathbb{P}\Big[\sup_{x\in[0, X_{\max})}|\Delta\partial_x^pu(x,t_n)|>\sqrt{\frac{\varepsilon}{3}}\Big]+\mathbb{P}\Big[\sup_{x\in[0, X_{\max})}|\Delta\partial_x^qu(x,t_n)|>\sqrt{\frac{\varepsilon}{3}}\Big],
\end{align*}
hence the results follow from Lemma~\ref{Lemm.Dx}.
\end{proof}
As for higher degree terms, we can take the similar approach to obtain general results but with more complicated notations. In this work, we focus on demonstrating the essence without involving more indices.

\subsection{Simplification on the Probability Bounds}
Before proceeding further, we simplify the expressions for $\varepsilon_{N}^*$ as well as the probability bounds in Lemma~\ref{Lemm.Dt} by considering the window width $h_N$ and the diverging sequence $B_N$ as follows
\begin{align*}
h_N = \frac{1}{N^a},~B_N = N^b\;.
\end{align*}
Here $a,b>0$ are positive coefficients to be determined. 

Consequently, we update the expressions of the five terms whose maximum defines the threshold $\varepsilon^*_N$ 

\begin{align*}
    &E_1(N) = \frac{3|C^*(X_i)|}{N^{{ 2}a}}, \;\;\;
    E_2(N) = \frac{6\mathcal{K}_{\max}^*(N^b+\|u\|_{L^{\infty}(\Omega)})}{N^{1-2a}},\;\;\;
    E_3(N) = \frac{6A(X_i)}{N^{-a}\big( N^{b}+\|u\|_{L^{\infty}(\Omega)} \big)}	\nonumber\\
    &E_4(N) =\frac{6\mathcal{K}_{\max}^*(N^b+\|u\|_{L^\infty(\Omega)})(a_0\ln N+r)\ln N}{N^{1-2a}},  \;\;\;
    E_5(N) = 12\sqrt{\alpha}d_0(X_i)\sqrt{\frac{a\ln N}{N^{1-{ 3}a}}}\;.
\end{align*}
When $N$ is sufficiently large, to determine $\varepsilon_N^*$, we only need to focus on comparing the powers of $N$ in $E_i(N)$, $i=1,2,\cdots,5$; this immediately leads to:
\begin{align*}
    E_2(N)= \mathcal{O}\left(E_4(N)\right), \;
\end{align*}
hence it's sufficient to only consider $E_1(N)$, $E_2(N)$, $E_4(N)$, and $E_5(N)$. The optimal choice of $a$ and $b$ is determined by requiring
\begin{align*}
    \begin{cases}
    	2a = 1-b-2a\\
    	2a = \frac{1-3a}{2}
    \end{cases}	\implies\begin{cases}
    a = \frac{1}{7}\\
    b = \frac{3}{7}
\end{cases}
\end{align*}

\newpage
\noindent To summarize the discussion above, we have
\begin{corollary}\label{cor_Dut}
Let  $h_N=N^{-1/7}$. For any $i=0,1,\dots,M$ and arbitrary real $r$, there exist finite positive constants $C^*(X_i),a_0,b_0,c_0,$ and $d_0(X_i)$ which do not depend on the temporal sample size $N$, such that for $N$ sufficiently large, any $\alpha>1$, and 
\begin{align*}
\varepsilon_N^*(X_i,r,\alpha)>N^{{-\frac{2}{7}}}\max\bigg\{3|C^*(X_i)|,6(a_0\ln N+r)\ln N,12\sqrt{\alpha}d_0(X_i)\sqrt{\frac{\ln N}{7}}\bigg\}\;,	
\end{align*}
we have:
\begin{align*}
\mathbb{P}\Big[\sup_{t\in[0,T]}|\Delta u_t(X_i,t)|>\varepsilon_N^*(X_i,r,\alpha)\Big]<2N\exp\left(-\frac{N^{{6}/7}}{2\sigma^2}\right)+b_0\exp(-c_0r)+4\sqrt{2}\eta^4N^{-\alpha/7}\;,
\end{align*}
\end{corollary}
Similarly, we can obtain optimal $w_M=M^{-1/(2p+{5})}$ and $B_M=M^{{(p+2)}/{(2p+5)}}$ for the estimation of $p$-th partial derivative of $u$. Consequently, the threshold lower bound in Lemma~\ref{Lemm.Dx} becomes
\begin{align*}
&\varepsilon_{M,p}^*(t_n,r,\alpha)>M^{-2/(2p+{5})}\max\bigg\{3|C^*(t_n)|,6p!(a_0\ln M+r)\ln M,12p!\sqrt{\alpha}d_0(t_n)\sqrt{\frac{\ln M}{2p+{5}}}\bigg\}\;.	
\end{align*}
Notice that the right hand side of the inequality above is non-decreasing with respect to $p\geq 0$. Moreover, note that for sufficiently large $M$, if the probability bound in Lemma~\ref{Lemm.Dx} holds for some $w_M$, then it holds for any smaller window width $w'_M<w_M$. Therefore, we have the following simplified result
\begin{corollary}\label{cor.singlespace}
Let $w_M=M^{-1/7}$. For any $n=0,1,\dots,N-1$ and arbitrary $r$, there exist finite positive constants $C^*(t_n),a_0,b_0,c_0,$ and $d_0(t_n)$ which do not depend on the spacial sample size $M$, such that for $M$ sufficiently large, any $\alpha>1$,  and 
\begin{align*}
&\varepsilon_{M}^*(t_n,r,\alpha)>M^{-2/(2P_{\max}+{5})}\max\bigg\{3|C^*(t_n)|,6P_{\max}!(a_0\ln M+r)\ln M,12P_{\max}!\sqrt{\alpha}d_0(t_n)\sqrt{\frac{\ln M}{2P_{\max}+{5}}}\bigg\}\;,	
\end{align*}
we have:
\begin{align*}
&\mathbb{P}\Big[\sup_{x\in[0, X_{\max})}|\widehat{\partial_x^pu}(x,t_n)-\partial_x^pu(x,t_n)|>\varepsilon_{M}^*\Big]\\
&<2M \exp\bigg(-\frac{M^{{(2P_{\max}+4)/(2P_{\max}+5)}}}{2\sigma^2}\bigg)+ b_0\exp(-c_0r)+4\sqrt{2}\eta^4M^{-\alpha/(2P_{\max}+{5})}\;
\end{align*}
for any order $ 0\leq p\leq P_{\max}$.
\end{corollary}

\noindent Similarly, for the product terms, we have
\begin{corollary}\label{cor.doublespace}
Let $w_M=M^{-1/7}$. For any $n=0,1,\dots,N-1$ and arbitrary $r$, there exist finite positive constants $C^*(t_n),a_0,b_0,c_0,$ and $d_0(t_n)$ which do not depend on the spacial sample size $M$, such that for $M$ sufficiently large, any $\alpha>1$,  and 
\begin{align*}
\varepsilon_{M}^{**}>\max\{3\|u(\cdot,t_n)\|_{P_{\max},\infty}\varepsilon^*_{M},~3(\varepsilon^*_{M})^2\}
\end{align*}
where $\|u(\cdot,t_n)\|_{P_{\max},\infty} = \sum_{0\leq k\leq P_{\max}}\|\partial_{x}^ku(\cdot,t_n)\|_{\infty}$, we have
\begin{align*}
&\frac{1}{4}\mathbb{P}\Big[\sup_{x\in[0, X_{\max})}|\widehat{\partial_x^pu}(x,t_n)\widehat{\partial_x^qu}(x,t_n)-\partial_x^pu(x,t_n)\partial_x^qu(x,t_n)|>\varepsilon_{M}^{**}\Big]\\
&<2M \exp\bigg(-\frac{M^{(2P_{\max}+4)/(2P_{\max}+5)}}{2\sigma^2}\bigg)+b_0\exp(-c_0r)+4\sqrt{2}\eta^4M^{-\alpha/(2P_{\max}+{5})}
\end{align*}
for any orders $0\leq p,q\leq P_{\max}$.
\end{corollary}
\subsection{$\ell_\infty$ Bound for the PDE Estimation Error $\btau$}
Notice that in the previous results, although the constants $C^*(X_i)$ and $d_0(X_i)$ are independent of $N$, they show dependence on the spacial point $X_i$. Similarly, $C^*(t_n)$ and $d_0(t_n)$ are independent of $M$, yet their values may depend on $N$. To guarantee that as both $N,M\to\infty$, these constants are uniformly bounded, we prove the following lemma.
\begin{lemma}\label{lemma_const}
For any integer $M\geq 1$, and any $i=0,1,\cdots, M-1$,  $|C^*(X_i)|$ and $d_0(X_i)$ in Corollary~\ref{cor_Dut} are bounded by constants that are independent of $M$. That is, there exist constants $C^*, d_0>0$ such that for any $M\geq 1$
\begin{align*}
\max_{i=0,\cdots, M-1}|C^*(X_i)|\leq C^*\|\partial_t^3u\|_\infty, \text{ and }
~\max_{i=0,\cdots, M-1}d_0(X_i)\leq d_0\;.
\end{align*}
\end{lemma}
\begin{proof}
From (3.7) in the Theorem 3.1 of~\cite{fan2018local}, we have
\begin{align*}
|C^*(X_i)|\leq C^*\|\partial_t^3u\|_\infty<\infty
\end{align*}
where $C^*$ only depends on the choice of the kernel function and the order of the Local-Polynomial. Recalling that $d_0(X_i)=16S^{1/2}\int|\zeta|^{1/2}|d\mathcal{K}^*(\zeta)|$ where $S=\sup_z\int y^2f(z,y|X_i)\,dy$. For a general real number $s$, we know that
\begin{align*}
&\sup_{z\in [0,T_{\max}]}\int |y|^sf(z,y|X_i)\,dy=\sup_{z\in [0,T_{\max}]}\int |y|^s\frac{1}{\sqrt{2\pi\sigma^2}}\exp\bigg(-\frac{(y-u(X_i,z))^2}{2\sigma^2}\bigg)\,dy\\
&=\sup_{z\in [0,T_{\max}]}\sigma^s2^{s/2}\frac{\Gamma\big(\frac{1+s}{2}\big)}{\sqrt{\pi}}{_1F_1}\bigg(-\frac{s}{2},\frac{1}{2},-\frac{1}{2}\Big(\frac{u(X_i,z)}{\sigma}\Big)^2\bigg)
\end{align*}
where ${_1F_1}(p,q,w)$ is Kummer's confluent hyper-geometric function of $w\in\mathbb{C}$ with parameters $p,q\in\mathbb{C}$~(See, e.g.\cite{winkelbauer2012moments}) and $\Gamma$ is the Gamma function. Since ${_1F_1}(-\frac{s}{2},\frac{1}{2},\cdot)$ is an entire function for fixed parameters,
\begin{align*}
\sup_{z\in [0,T_{\max}]}\int |y|^sf(z,y|X_i)\,dy\leq 	\sup_{z\in [0,T_{\max}]}\sigma^s2^{s/2}\frac{\Gamma\big(\frac{1+s}{2}\big)}{\sqrt{\pi}}\sup_{w\in [-\frac{\max_{x\in\Omega} u^2(x,z)}{2\sigma^2},-\frac{\min_{x\in\Omega} u^2(x,z)}{2\sigma^2}]}{_1F_1}(-\frac{s}{2},\frac{1}{2},w)<\infty
\end{align*}
which clearly does not depend on $M$. Taking $s=2$, we can obtain that $d_0(X_i)\leq d_0$ for some $d_0$ that only depends on the choice of  kernel $\mathcal{K}$, underlying function $\|u\|_{L^\infty(\Omega)}$, and noise level $\sigma$.
\end{proof}
Note that the same proof can derive that the constants in Lemma~\ref{Lemm.Dx} and Lemma~\ref{Lemm.Dxx} are also bounded by $N$-independent constants. This technical lemma allows us to state
\begin{proposition}\label{prop2}
Take $h_N=N^{-1/7}$ in the temporal direction and $w_{M} = M^{-1/7}$ in the space direction. There exist constants $C$, $a_0$, $b_0$, and $c_0$ which do not depend on $N$ nor $M$ such that for $N$ and $M$ sufficiently large, any $r$, $\alpha>1$, and 
\begin{align*}
    \varepsilon_{N,M}(r,\alpha)>C\max\bigg\{\frac{(a_0\ln N+r)\ln N}{N^{{2}/7}}, \frac{\sqrt{\alpha\ln N}}{N^{{2}/7}},\frac{(a_0\ln M+r)\ln M}{M^{2/(2P_{\max}+{5})}}, \sqrt{\frac{\alpha\ln M}{(2P_{\max}+{5})M^{4/(2P_{\max}+{5})}}}\bigg\}	
\end{align*}
we have
\begin{align*}
&\mathbb{P}\Big[\|\btau\|_\infty>\varepsilon_{N,M}\Big]<\\ &2NM\exp\left(-\frac{N^{{6}/7}}{2\sigma^2}\right)+b_0\exp(-c_0r)M+4\sqrt{2}\eta^4MN^{-\alpha/7}+\\
&8{s}NM \exp\bigg(-\frac{M^{{(2P_{\max}+4)/(2P_{\max}+5)}}}{2\sigma^2}\bigg)+ 4s b_0\exp(-c_0r)N+16\sqrt{2}\eta^4sNM^{-\alpha/(2P_{\max}+5)}
\end{align*}
Here $K$ is the number of feature variables in the dictionary.
\end{proposition}

\begin{proof}
By triangle inequality, the $\ell_\infty$-norm of PDE estimation error $\mathbf{\tau}$~\eqref{residual} can be bounded by
\begin{align*}
\|\btau\|_\infty&\leq \|\Delta \bF\bbeta^*\|_\infty+\|\Delta \bu_t\|_\infty\;.
\end{align*}
By Corollary~\ref{cor_Dut} and Lemma~\ref{lemma_const}, there exists a constant $C_1$ independent of $N$ and $M$ such that with  sufficiently large $N$ and any $\varepsilon_{N}(r,\alpha)>C_1 N^{-{2}/7}\max\{(a_0\ln N+r)\ln N, \sqrt{\alpha\ln N}\}$, we have
\begin{align*}
\mathbb{P}\Big[\|\Delta \bu_t\|_\infty>\varepsilon_{N}(r,\alpha)\Big]&\leq\mathbb{P}\Big[\max_{i=0,1,\cdots,M-1}\sup_{t\in[0,T_{\max}]}|\Delta u_t(X_i,t)|>\varepsilon_{N}(r,\alpha)\Big]	\\
&\leq \sum_{i=0}^{M-1}\mathbb{P}\Big[\sup_{t\in[0,T_{\max}]}|\Delta u_t(X_i,t)|>\varepsilon_{N}(r,\alpha)\Big]\\
&<2NM\exp\left(-\frac{N^{{ 6}/7}}{2\sigma^2}\right)+ b_0\exp(-c_0r)M+4\sqrt{2}\eta^4MN^{-\alpha/7}\;.
\end{align*}
On the other hand, if we denote $\Delta F_k(x,t)$ as the approximation error of the $k$-th feature variable at time $t$ and space $x$, we have
\begin{align*}
\|\Delta \bF\bbeta^*\|_\infty\leq \max_{n=0,1,\cdots, N}\|\bbeta^*\|_\infty	\sup_{x\in[0, X_{\max})}\sum_{k=1}^s|\Delta F_{k}(x,t_n)|\;.
\end{align*}
By Corollary~\ref{cor.singlespace} and~\ref{cor.doublespace}, there exists a constant $C_2$ independent of $N$ and $M$ such that with sufficiently large $M$ and any $\varepsilon_{K,M}(r,\alpha)>C_2P_{\max}!K\|\bbeta^*\|_{\infty}M^{-2/(2P_{\max}+{ 5})}\max\{(a_0\ln M+r)\ln M,\sqrt{\frac{\alpha\ln M}{2P_{\max}+{ 5}}}\}$, we have
\begin{align*}
&\mathbb{P}\Big[\|\Delta \bF\bbeta^*\|_\infty>\varepsilon_{M}(r,\alpha)\Big]\leq \sum_{n=0}^{N-1}\sum_{k=1}^{s}\mathbb{P}\Big[\sup_{x\in[0, X_{\max})}|\Delta F_{k}(x,t_n)|>\frac{\varepsilon_{M}(r,\alpha)}{s\|\bbeta^*\|_\infty}\Big]\\
&<8NM {s} \exp\bigg(-\frac{{  M^{(2P_{\max}+4)/(2P_{\max}+5)}}}{2\sigma^2}\bigg)+ 4b_0\exp(-c_0r)Ns+16\sqrt{2}\eta^4NsM^{-\alpha/(2P_{\max}+{ 5})}\;.
\end{align*}
Taking $C=\max\{2C_1,2s\|\bbeta^*\|_\infty C_2P_{\max}!\}$   proves the theorem.
\end{proof}
\subsection{Further Simplification} \label{simplification}
We further simplify our result by taking $M=N^b$ for some coefficient $b>0$. Since $r$ and $\alpha$ are arbitrary, we can vary  them as we increase $M,N$ by taking $r=N^c$ and $\alpha=N^d$ for some positive coefficients $c>0$ and $d>0$, respectively. Consequently, we have the lower bound for $\varepsilon_{N,M}$ in Proposition~\ref{prop2} becoming
\begin{align}
\varepsilon_{N,M}(r,\alpha)>C\max\bigg\{\frac{(a_0\ln N+N^c)\ln N}{N^{{ 2}/7}}, \frac{\sqrt{\ln N}}{N^{{ 2}/7-d/2}},\frac{b(a_0b\ln N+N^c)\ln N}{N^{2b/(2P_{\max}+{ 5})}}, \sqrt{\frac{b\ln N}{(2P_{\max}+{ 5})N^{4b/(2P_{\max}+{ 5})-d}}}\bigg\}\;,\label{eq_epsilon2}
    \end{align}
To guarantee that the lower bound~\eqref{eq_epsilon2} converges to $0$ as $N\to\infty$, we have the following constraints on positive coefficients $b,c$, and $d$
\begin{align*}
\begin{cases}
0<c<{ 2}/7\\
{ 2}/7-d/2>0\\
c<2b/(2P_{\max}+{ 5})\\
4b/(2P_{\max}+{ 5})-d>0
\end{cases}
\end{align*}
Furthermore, we take $d=2c$ so that
\begin{align*}
     \frac{\sqrt{\ln N}}{N^{{ 2}/7-d/2}}=\mathcal{O}\left(\frac{(a_0\ln N+N^c)\ln N}{N^{{ 2}/7}}\right),\;\;\;\sqrt{\frac{b\ln N}{N^{4b/(2P_{\max}+{ 5})-d}}}=\mathcal{O}\left(\frac{b(a_0b\ln N+N^c)\ln N}{N^{2b/(2P_{\max}+{ 5})}} \right)\;.
\end{align*}
and we can focus on the second and fourth term in~\eqref{eq_epsilon2}. As a result, the optimal choice for $b$ is computed by ${ 2}/7=2b/(2P_{\max}+{ 5})\implies b=({ 2}P_{\max}+{ 5 })/7$.
Based on the set-ups above, we obtain that for $N$ sufficiently large, with
\begin{align*}
\varepsilon_{N}(c)> C \frac{\ln N}{N^{{ 2}/7-c}}
\end{align*}
for any $0<c<{ 2}/7$,  we have
\begin{align*}
    \mathbb{P}\Big[\|&\btau\|_\infty>\varepsilon_{N}(c)\Big]<\\ 
    &2N^{{ (2P_{\max}+12)}/7}\exp\left(-\frac{N^{{ 6}/7}}{2\sigma^2}\right)
    +b_0\exp(-c_0N^c)N^{{ (2P_{\max}+5)/7}}+4\sqrt{2}\eta^4N^{-N^{2c}/{ 7}} + \\
    &8N^{{ (2P_{\max}+12)}/7}{ K} \exp\left(-\frac{N^{({ 2}P_{\max}+{ 5})/7}}{2\sigma^2}\right)+ 4b_0\exp(-c_0N^c)NK+16\sqrt{2}\eta^4KN^{-N^{2c}/7}\\
    &=
    \mathcal{O}\left(N^{{\frac{2P_{\max}+5}{7}}}\exp\bigg(-\frac{1}{6}N^{c}\bigg) \right),
\end{align*}
where in the last equality, we plug $b_{0}=2$ and $c_{0}=\frac{1}{6}$ from ~\cite{bretagnolle1989hungarian}.
Combining this with Lemma~\ref{Zlemma} proves the first part of the Proposition~\ref{prop1}.

\subsection{Proof of $\ell_{\infty}$ bound in \eqref{infbound1}} \label{B6}
Recall that in \eqref{eq_beta_min}, we have 
\begin{equation*}
    \check{\bbeta}_{\mathcal{S}}-\bbeta^*_{\mathcal{S}}
    = \big( \widehat{\bF}^T_S\widehat{\bF}_\mS \big)^{-1}
    \bigg( \widehat{\bF}_{\mS}^T(\Delta \bu_t - \Delta \bF_\mS\bbeta_\mS^*) - \lambda_{N}NM\check{\mathbf{z}}_\mS \bigg).
\end{equation*}
Now, we are ready to bound the $\left\VERT \widehat{\bbeta}_\mS^\lambda-\bbeta_\mS^*\right\VERT \ell_{\infty}$ bound in \eqref{infbound1}
as follows:
\begin{align*}
\max_{k\in \mS}|\bbeta_k-\bbeta^*_k|&\leq	\left\VERT\Big(\widehat{\bF}_S^T\widehat{\bF}_\mS\Big)^{-1}\right\VERT_{2}\|\widehat{\bF}_\mS^T\btau\|_\infty+\lambda NM\left\VERT\Big(\widehat{\bF}_\mS^T\widehat{\bF}_\mS\Big)^{-1}\right\VERT_{2}\\
&\leq \left\VERT\Big(\widehat{\bF}_\mS^T\widehat{\bF}_\mS/(NM)\Big)^{-1}\right\VERT_{2}\,\bigg(\|\widehat{\bF}_\mS^T\btau\|_\infty/(NM)+\lambda\bigg)\\
&\overset{\eqref{assum1}}{\leq} \sqrt{K}C_{\min}\,\bigg(\|\widehat{\bF}_\mS^T\btau\|_\infty/(NM)+\lambda\bigg)\\
&\leq \sqrt{K}C_{\min}\,\bigg(\|\btau\|_\infty\frac{\left\VERT\widehat{\bF}_\mS\right\VERT_{\infty,\infty}}{NM}+\lambda\bigg)\\
&\leq \sqrt{K}C_{\min}\,\bigg(\|\btau\|_\infty\frac{\left\VERT\widehat{\bF}\right\VERT_F}{\sqrt{NM}}+\lambda\bigg)\\
&{\leq} \sqrt{K}C_{\min}\,\bigg(K\|\btau\|_\infty+\lambda\bigg),
\end{align*}
where we use normalized columns of $\widehat{\bF}$ in the last inequality.
Following the set-ups from Proposition~\ref{prop1} gives the desired result.

\newpage
\section{Proofs of Lemmas~\ref{lem.eign'} and ~\ref{lem.incoh'}}

\begin{corollary}\label{cor.Dxxx}
Fix any four orders $p,q,k\geq 0$, and let $B_M$ be an arbitrary increasing sequence $B_M\to\infty$ as $M\to\infty$, and  $B^{'}_M=B_M+\|u\|_{L^\infty(\Omega)}$. For any $n=0,1,\dots,N-1$ and arbitrary $r$, there exist finite positive constants $A(t_n),C^*(t_n),a_0,b_0,c_0,$ and $d_0(t_n)$ which do not depend on the spacial sample size $M$, such that for any $\alpha>1$ and 
\begin{align*} 
\varepsilon_{M,p,q,k}^{***}>\max\bigg\{3\|\partial_x^ku(\cdot,t_n)\|_\infty\varepsilon^{**}_{M,p,q},~3\|\partial_x^pu(\cdot,t_n)\partial_x^qu(\cdot,t_n)\|_\infty\varepsilon^{**}_{M,k},~3(\varepsilon^{**}_{M,p,q})^2,~3(\varepsilon^{**}_{M,k})^2\bigg\}
\end{align*}
as long as $M>M(\alpha)$ for some positive integer $M(\alpha)$,  we have:
\begin{align*}
\frac{1}{4}\mathbb{P}&\Bigg[\sup_{x\in[0, X_{\max})}\left|\widehat{\partial_x^pu}(x,t_n)\widehat{\partial_x^qu}(x,t_n)\widehat{\partial_x^ku}(x,t_n)-\partial_x^pu(x,t_n)\partial_x^qu(x,t_n)\partial_x^ku(x,t_n)\right|>\varepsilon_{M,p,q,k}^{***}\Bigg] \nonumber\\
&< 8M \exp\bigg(-\frac{M^{(2P_{\max}+4)/(2P_{\max}+5)}}{2\sigma^2}\bigg)+4b_0\exp(-c_0r)+16\sqrt{2}\eta^4M^{-\alpha/(2P_{\max}+{5})}\;,
\end{align*}	
Here $\varepsilon^{**}_{M,p,q}$ and $\varepsilon^{**}_{M,k,l}$ (depending on $B_M'$) are the thresholds in Corollary~\ref{cor.doublespace} for the sup-norm bound of the estimator $\widehat{\partial_x^pu}\widehat{\partial_x^qu}$ and $\widehat{\partial_x^ku}\widehat{\partial_x^lu}$, respectively,
\end{corollary}

\begin{proof}
Notice that for any $\varepsilon>0$, we can bound the probability:
\begin{align*}
&\mathbb{P}\Bigg[\sup_{x\in[0, X_{\max})}\left|\widehat{\partial_x^pu}(x,t_n)\widehat{\partial_x^qu}(x,t_n)\widehat{\partial_x^ku}(x,t_n)-\partial_x^pu(x,t_n)\partial_x^qu(x,t_n)\partial_x^ku(x,t_n)\right|>\varepsilon\Bigg] \\
&\leq\mathbb{P}\Bigg[\|\partial_x^ku(\cdot,t_n)\|_{\infty}\sup_{x\in[0, X_{\max})}\left|\widehat{\partial_x^pu}(x,t_n)\widehat{\partial_x^qu}(x,t_n)-\partial_x^pu(x,t_n)\partial_x^qu(x,t_n)\right|>\varepsilon/3\Bigg] \\
&+\mathbb{P}\Bigg[\|\partial_x^pu(\cdot,t_n)\partial_x^qu(\cdot,t_n)\|_{\infty}\sup_{x\in[0, X_{\max})}\left|\widehat{\partial_x^ku}(x,t_n)-\partial_x^ku(x,t_n)\right|>\varepsilon/3\Bigg] \\
&+ \mathbb{P}\Big[\sup_{x\in[0, X_{\max})}\left|\widehat{\partial_x^pu}(x,t_n)\widehat{\partial_x^qu}(x,t_n)-\partial_x^pu(x,t_n)\partial_x^qu(x,t_n)\right|>\sqrt{\frac{\varepsilon}{3}}\Big]\\ 
&+ \mathbb{P}\Big[\sup_{x\in[0, X_{\max})}\left|\widehat{\partial_x^ku}(x,t_n)-\partial_x^ku(x,t_n)\right|>\sqrt{\frac{\varepsilon}{3}}\Big],
\end{align*}
hence the results follow from corolloary~\ref{cor.doublespace}.
\end{proof}

\begin{corollary}\label{cor.Dxxxx}
Fix any four orders $p,q,k,l\geq 0$, and let $B_M$ be an arbitrary increasing sequence $B_M\to\infty$ as $M\to\infty$, and  $B^{'}_M=B_M+\|u\|_{L^\infty(\Omega)}$. For any $n=0,1,\dots,N-1$ and arbitrary $r$, there exist finite positive constants $A(t_n),C^*(t_n),a_0,b_0,c_0,$ and $d_0(t_n)$ which do not depend on the spacial sample size $M$, such that for any $\alpha>1$ and 
\begin{align*} 
\varepsilon_{M,p,q,k,l}^{****}>\max\bigg\{3\|\partial_x^pu(\cdot,t_n)\partial_x^qu(\cdot,t_n)\|_\infty\varepsilon^{**}_{M,p,q},~3\|\partial_x^ku(\cdot,t_n)\partial_x^lu(\cdot,t_n)\|_\infty\varepsilon^{**}_{M,k,l},~3(\varepsilon^{**}_{M,p,q})^2,~3(\varepsilon^{**}_{M,k,l})^2\bigg\}
\end{align*}
as long as $M>M(\alpha)$ for some positive integer $M(\alpha)$,  we have:
\begin{align*}
\frac{1}{4}\mathbb{P}&\Bigg[\sup_{x\in[0, X_{\max})}\left|\widehat{\partial_x^pu}(x,t_n)\widehat{\partial_x^qu}(x,t_n)\widehat{\partial_x^ku}(x,t_n)\widehat{\partial_x^lu}(x,t_n)-\partial_x^pu(x,t_n)\partial_x^qu(x,t_n)\partial_x^ku(x,t_n)\partial_x^lu(x,t_n)\right|>\varepsilon_{M,p,q,k,l}^{****}\Bigg] \nonumber\\
&< 8M \exp\bigg(-\frac{M^{(2P_{\max}+4)/(2P_{\max}+5)}}{2\sigma^2}\bigg)+4b_0\exp(-c_0r)+16\sqrt{2}\eta^4M^{-\alpha/(2P_{\max}+{5})}\;,
\end{align*}	
Here $\varepsilon^{**}_{M,p,q}$ and $\varepsilon^{**}_{M,k,l}$ (depending on $B_M'$) are the thresholds in Corollary~\ref{cor.doublespace} for the sup-norm bound of the estimator $\widehat{\partial_x^pu}\widehat{\partial_x^qu}$ and $\widehat{\partial_x^ku}\widehat{\partial_x^lu}$, respectively,
\end{corollary}

\begin{proof}
Notice that for any $\varepsilon>0$, we can bound the probability:
\begin{align*}
&\mathbb{P}\Bigg[\sup_{x\in[0, X_{\max})}\left|\widehat{\partial_x^pu}(x,t_n)\widehat{\partial_x^qu}(x,t_n)\widehat{\partial_x^ku}(x,t_n)\widehat{\partial_x^lu}(x,t_n)-\partial_x^pu(x,t_n)\partial_x^qu(x,t_n)\partial_x^ku(x,t_n)\partial_x^lu(x,t_n)\right|>\varepsilon\Bigg] \\
&\leq\mathbb{P}\Bigg[\|\partial_x^ku(\cdot,t_n)\partial_x^lu(\cdot,t_n)\|_{\infty}\sup_{x\in[0, X_{\max})}\left|\widehat{\partial_x^pu}(x,t_n)\widehat{\partial_x^qu}(x,t_n)-\partial_x^pu(x,t_n)\partial_x^qu(x,t_n)\right|>\varepsilon/3\Bigg] \\
&+\mathbb{P}\Bigg[\|\partial_x^pu(\cdot,t_n)\partial_x^qu(\cdot,t_n)\|_{\infty}\sup_{x\in[0, X_{\max})}\left|\widehat{\partial_x^ku}(x,t_n)\widehat{\partial_x^lu}(x,t_n)-\partial_x^ku(x,t_n)\partial_x^lu(x,t_n)\right|>\varepsilon/3\Bigg] \\
&+ \mathbb{P}\Big[\sup_{x\in[0, X_{\max})}\left|\widehat{\partial_x^pu}(x,t_n)\widehat{\partial_x^qu}(x,t_n)-\partial_x^pu(x,t_n)\partial_x^qu(x,t_n)\right|>\sqrt{\frac{\varepsilon}{3}}\Big]\\ 
&+ \mathbb{P}\Big[\sup_{x\in[0, X_{\max})}\left|\widehat{\partial_x^ku}(x,t_n)\widehat{\partial_x^lu}(x,t_n)-\partial_x^ku(x,t_n)\partial_x^l0u(x,t_n)\right|>\sqrt{\frac{\varepsilon}{3}}\Big],
\end{align*}
hence the results follow from corolloary~\ref{cor.doublespace}.
\end{proof}

\begin{lemma} \label{lemm.Fsc_conv}
Let $\varepsilon_{M}^*,\varepsilon_{M}^{**},\varepsilon_{M}^{***},\varepsilon_{M}^{****}$ be the thresholds defined in corollaries~\ref{cor.singlespace},~\ref{cor.doublespace},~\ref{cor.Dxxx}, and ~\ref{cor.Dxxxx}. 
Then for any $\varepsilon_{M}^{\text{max}^{'}}$ such that
\begin{equation*}
\varepsilon_{M}^{\text{max}^{'}}>\sqrt{s(K-s)}\max\bigg\{\varepsilon_{M}^*,\varepsilon_{M}^{**},\varepsilon_{M}^{***},\varepsilon_{M}^{****}\bigg\}\;,
\end{equation*}
then, for $0<c<\frac{2}{7}$, and for sufficiently large enough $N$, we have 
\begin{align*}
    \mathbb{P}\Bigg[ \frac{1}{NM}\left\VERT \widehat{\bF}_{\mS^{c}}^{T}\widehat{\bF}_{\mS}-{\bF}_{\mS^{c}}^{T}{\bF}_{\mS}
   \right\VERT_{2}  > \varepsilon_{M}^{\text{max}^{'}} \Bigg] \leq \mathcal{O}\bigg(N\exp\big(-\frac{1}{6}N^c\big)\bigg).
\end{align*}
\end{lemma}

\begin{proof}
\begin{align*}
    \mathbb{P}\Bigg[&\frac{1}{NM}\left\VERT \widehat{\bF}_{\mS^{c}}^{T}\widehat{\bF}_{\mS}-{\bF}_{\mS^{c}}^{T}{\bF}_{\mS} \right\VERT_{2}>\varepsilon_{M}^{\text{max}^{'}}\Bigg]\\
    &\leq \mathbb{P}\Bigg[\left\VERT\widehat{\bF}_{\mS^{c}}^{T}\widehat{\bF}_{\mS}-{\bF}_{\mS^{c}}^{T}{\bF}_{\mS}\right\VERT_{\text{F}}>NM\varepsilon_{M}^{\text{max}^{'}}\Bigg]\\
    &\leq \mathbb{P}\Bigg[\left\VERT\widehat{\bF}_{\mS^{c}}^{T}\widehat{\bF}_{\mS}-{\bF}_{\mS^{c}}^{T}{\bF}_{\mS}\right\VERT_{\infty,\infty}>NM\frac{\varepsilon_{M}^{\text{max}^{'}}}{\sqrt{s(K-s)}}\Bigg]\\
    &\leq \mathbb{P}\Bigg[\max_{n=0,\dots,N-1}\sup_{x\in[0, X_{\max})}\left|\widehat{\bF_{i}}(x,t_n)\widehat{\bF_{j}}(x,t_n)-\bF_{i}(x,t_n) \bF_{j}(x,t_n)\right|>\frac{\varepsilon_{M}^{\text{max}^{'}}}{\sqrt{s(K-s)}}\Bigg] \\
    &\leq \sum_{n=0}^{N-1}\mathbb{P}\Bigg[\sup_{x\in[0, X_{\max})}\left|\widehat{\bF_{i}}(x,t_n)\widehat{\bF_{j}}(x,t_n)-\bF_{i}(x,t_n) \bF_{j}(x,t_n)\right|>\frac{\varepsilon_{M}^{\text{max}^{'}}}{\sqrt{s(K-s)}}\Bigg] \\
    &\leq \mathcal{O}\bigg(N\exp\big(-\frac{1}{6}N^c\big)\bigg)\;,
\end{align*}
where we use the results from corollaries~\ref{cor.singlespace},~\ref{cor.doublespace},~\ref{cor.Dxxx}, and ~\ref{cor.Dxxxx}, and simplication argument used in the Appendix~\ref{simplification} in the last inequality.
\end{proof}

\subsection{Proof of Lemma~\ref{lem.eign'}} \label{proof_eigen'}
\begin{proof}
Observe that we can write:
\begin{align*}
    \Lambda_{\text{min}}\bigg(\frac{1}{NM} \bF_{\mS}^{T}{\bF_{\mS}}\bigg)
    &:=\frac{1}{NM} \min_{\|x\|_{2}=1}x^{T}\bigg(\bF_{\mS}^{T}\bF_{\mS}\bigg)x\\
    &= \frac{1}{NM} \min_{\|x\|_{2}=1} \bigg\{ x^{T}\bigg(\widehat{\bF}_{\mS}^{T}\widehat{\bF}_{\mS}\bigg)x + 
    x^{T}\bigg(\bF_{\mS}^{T}\bF_{\mS}-\widehat{\bF}_{\mS}^{T}\widehat{\bF}_{\mS}\bigg)x \bigg\}\\
    &\leq \frac{1}{NM} \bigg\{ y^{T}\bigg(\widehat{\bF}_{\mS}^{T}\widehat{\bF}_{\mS}\bigg)y + 
    y^{T}\bigg(\bF_{\mS}^{T}\bF_{\mS}-\widehat{\bF}_{\mS}^{T}\widehat{\bF}_{\mS}\bigg)y \bigg\}
\end{align*}
where $y\in\mathbb{R}^{K}$ is a unit-norm minimal eigen-vector of $\frac{1}{NM} \bF_{\mS}^{T}\bF_{\mS}$. 
Therefore, we can write,
\begin{align*}
    \Lambda_{\text{min}}\bigg(\frac{1}{NM} \widehat{\bF}_{\mS}^{T}\widehat{\bF}_{\mS}\bigg)
    &\geq \Lambda_{\text{min}}\bigg(\frac{1}{NM} \bF_{\mS}^{T}\bF_{\mS}\bigg)-\frac{1}{NM}\left\VERT \bF_{\mS}^{T}\bF_{\mS}-\widehat{\bF}_{\mS}^{T}\widehat{\bF}_{\mS} \right\VERT_{2} \\
    & \geq C_{\text{min}}-\frac{1}{NM}\left\VERT \widehat{\bF}_{\mS}^{T}\widehat{\bF}_{\mS}-{\bF}_{\mS}^{T}{\bF}_{\mS} \right\VERT_{2}.
\end{align*}
By using a similar argument used in Lemma~\ref{lemm.Fsc_conv}, we can prove $\frac{1}{NM}\left\VERT \widehat{\bF}_{\mS}^{T}\widehat{\bF}_{\mS}-{\bF}_{\mS}^{T}{\bF}_{\mS} \right\VERT_{2}\rightarrow{0}$ with high-probability as $N\rightarrow{\infty}$.
For any $\varepsilon_{M}^{\text{max}}$ such that,
\begin{equation*}
\varepsilon_{M}^{\text{max}}>s\max\bigg\{\varepsilon_{M}^*,\varepsilon_{M}^{**},\varepsilon_{M}^{***},\varepsilon_{M}^{****}\bigg\}\;,
\end{equation*}
Then, we can bound the probability as follows:
\begin{align*}
\mathbb{P}\Bigg[&\frac{1}{NM}\left\VERT\widehat{\bF}_{\mS}^{T}\widehat{\bF}_{\mS}-{\bF}_{\mS}^{T}{\bF}_{\mS}\right\VERT_2>\varepsilon_{M}^{\text{max}}\Bigg]\\
    &\leq \mathbb{P}\Bigg[\left\VERT\widehat{\bF}_{\mS}^{T}\widehat{\bF}_{\mS}-{\bF}_{\mS}^{T}{\bF}_{\mS}\right\VERT_{\text{F}}>NM\varepsilon_{M}^{\text{max}}\Bigg]
    \leq \mathbb{P}\Bigg[\left\VERT\widehat{\bF}_{\mS}^{T}\widehat{\bF}_{\mS}-{\bF}_{\mS}^{T}{\bF}_{\mS}\right\VERT_{\infty,\infty}>NM\frac{\varepsilon_{M}^{\text{max}}}{s}\Bigg]\\
    &\leq \mathbb{P}\Bigg[\max_{n=0,\dots,N-1}\sup_{x\in[0, X_{\max})}\left|\widehat{\bF_{i}}(x,t_n)\widehat{\bF_{j}}(x,t_n)-\bF_{i}(x,t_n) \bF_{j}(x,t_n)\right|>\frac{\varepsilon_{M}^{\text{max}}}{s}\Bigg] \\
    &\leq \sum_{n=0}^{N-1}\mathbb{P}\Bigg[\sup_{x\in[0, X_{\max})}\left|\widehat{\bF_{i}}(x,t_n)\widehat{\bF_{j}}(x,t_n)-\bF_{i}(x,t_n) \bF_{j}(x,t_n)\right|>\frac{\varepsilon_{M}^{\text{max}}}{s}\Bigg] \\
    &\leq \mathcal{O}\bigg(N\exp\big(-\frac{1}{6}N^c\big)\bigg)\;.
\end{align*}
\end{proof}

\subsection{Proof of Lemma~\ref{lem.incoh'}} 
\begin{proof}
Motviated from ~\cite{ravikumar2010high}, we begin the proof by decomposing the sample matrix $\big( \widehat{\bF}_{\mS^{c}}^{T}\widehat{\bF}_{\mS} \big)\big( \widehat{\bF}_{\mS}^{T}\widehat{\bF}_{\mS} \big)^{-1}$ into four parts: 
\begin{align*}
    \big( \widehat{\bF}_{\mS^{c}}^{T}\widehat{\bF}_{\mS} \big)\big( \widehat{\bF}_{\mS}^{T}\widehat{\bF}_{\mS} \big)^{-1}
    &= \underbrace{\bF_{\mS^{c}}^{T}\bF_{\mS}\bigg(\big(\widehat{\bF}_{\mS}^{T}\widehat{\bF}_{\mS} \big)^{-1}-\big( \bF_{\mS}^{T}\bF_{\mS} \big)^{-1} \bigg)}_{\coloneqq\bf{T_{1}}} + 
    \underbrace{\bigg( \widehat{\bF}_{\mS^{c}}^{T}\widehat{\bF}_{\mS}-{\bF}_{\mS^{c}}^{T}{\bF}_{\mS} \bigg)\big( {\bF}_{\mS}^{T}{\bF}_{\mS} \big)^{-1}}_{\coloneqq\bf{T_{2}}} \\
    &+ \underbrace{\bigg( \widehat{\bF}_{\mS^{c}}^{T}\widehat{\bF}_{\mS}-{\bF}_{\mS^{c}}^{T}{\bF}_{\mS} \bigg)\bigg( \big(\widehat{\bF}_{\mS}^{T}\widehat{\bF}_{\mS} \big)^{-1}-\big( \bF_{\mS}^{T}\bF_{\mS} \big)^{-1} \bigg)}_{\coloneqq\bf{T_{3}}} \\
    &+ \underbrace{\big( {\bF}_{\mS^{c}}^{T}{\bF}_{\mS} \big)\big( {\bF}_{\mS}^{T}{\bF}_{\mS} \big)^{-1} }_{\coloneqq\bf{T_{4}}}.
\end{align*}
Since we know $\VERT {\bf{T_{4}}} \VERT_{\infty} \leq 1-\mu$ for some $\mu\in(0,1]$, the decomposition reduces the proof showing $\VERT \bf{T_{i}} \VERT_{\infty}\rightarrow{0}$ with probability $1-\mathcal{O}(N\exp(-\frac{1}{6}N^c))$ for $i=1,2,3$. \\ \\

\noindent\textbf{\textit{1. Control of $\bf{T_{1}}$}}: 
Observe that we can re-factorize $\bf{T_{1}}$ as follows: 
\begin{equation*}
    \bf{T_{1}}=\big( {\bF}_{\mS^{c}}^{T}{\bF}_{\mS} \big)\big( {\bF}_{\mS}^{T}{\bF}_{\mS} \big)^{-1}
    \big[ F_{\mS}^{T}F_{\mS} - \widehat{F}_{\mS}^{T}\widehat{F}_{\mS} \big]
    \big( \widehat{F}_{\mS}^{T}\widehat{F}_{\mS}\big)^{-1}.
\end{equation*}
Then, by taking the advantage of sub-multiplicative property $\VERT AB \VERT_{\infty} \leq \VERT A \VERT_{\infty} \VERT B \VERT_{\infty}$ and the fact 
$\VERT {\bf{T_{4}}} \VERT_{\infty} \leq 1-\mu$ and $\VERT C \VERT_{\infty}\leq \sqrt{N} \VERT C \VERT_{2}$ for $C \in \mathbb{R}^{M \times N}$, we can bound $\VERT \bf{T_{1}} \VERT_{\infty}$ as follows:
\begin{align*}
    \left\VERT \bf{T_{1}} \right\VERT_{\infty}
    &\leq \left\VERT \big( {\bF}_{\mS^{c}}^{T}{\bF}_{\mS} \big)\big( {\bF}_{\mS}^{T}{\bF}_{\mS} \big)^{-1} \right\VERT_{\infty}  
    \left\VERT \bF_{\mS}^{T}\bF_{\mS} - \widehat{\bF}_{\mS}^{T}\widehat{\bF}_{\mS} \right\VERT_{\infty}  
    \left\VERT \big( \widehat{\bF}_{\mS}^{T}\widehat{\bF}_{\mS}\big)^{-1} \right\VERT_{\infty} \\
    &\leq s(1-\mu)\bigg(\frac{1}{NM}\left\VERT \bF_{\mS}^{T}\bF_{\mS} - \widehat{\bF}_{\mS}^{T}\widehat{\bF}_{\mS} \right\VERT_{2} \bigg)
    \bigg( NM \left\VERT \big( \widehat{\bF}_{\mS}^{T}\widehat{\bF}_{\mS}\big)^{-1} \right\VERT_{2} \bigg) \\
    &\leq \frac{s(1-\mu)}{C_{\text{min}}}\bigg(\frac{1}{NM}\left\VERT \bF_{\mS}^{T}\bF_{\mS} - \widehat{\bF}_{\mS}^{T}\widehat{\bF}_{\mS} \right\VERT_{2} \bigg).
\end{align*}
Note that we use $\VERT \big( \widehat{\bF}_{\mS}^{T}\widehat{\bF}_{\mS}\big)^{-1}\VERT_{2}\leq \frac{1}{NMC_{min}}$ with probability $1- \mathcal{O}(N\exp(-\frac{1}{6}N^c))$ in the last inequality from Lemma~\ref{lem.eign'}.
\\ \\
\noindent\textbf{\textit{2. Control of $\bf{T_{2}}$}}: With similar techniques employed for controlling $\VERT \bf{T_{1}}\VERT_{\infty}$, we can bound $\VERT \bf{T_{2}}\VERT_{\infty}$ as follows:
\begin{align*}
    \left\VERT \bf{T_{2}} \right\VERT_{\infty}
    &\leq \left\VERT \widehat{\bF}_{\mS^{c}}^{T}\widehat{\bF}_{\mS}-{\bF}_{\mS^{c}}^{T}{\bF}_{\mS} \right\VERT_{\infty} 
    \left\VERT \big( {\bF}_{\mS}^{T}{\bF}_{\mS} \big)^{-1} \right\VERT_{\infty} \\
    &\leq s \left\VERT \widehat{\bF}_{\mS^{c}}^{T}\widehat{\bF}_{\mS}-{\bF}_{\mS^{c}}^{T}{\bF}_{\mS} \right\VERT_{2} 
    \left\VERT \big( {\bF}_{\mS}^{T}{\bF}_{\mS} \big)^{-1} \right\VERT_{2} \\
    &= s \bigg( \frac{1}{NM} \left\VERT \widehat{\bF}_{\mS^{c}}^{T}\widehat{\bF}_{\mS}-{\bF}_{\mS^{c}}^{T}{\bF}_{\mS} \right\VERT_{2} \bigg) 
    \bigg( NM \left\VERT \big( \widehat{\bF}_{\mS}^{T}\widehat{\bF}_{\mS}\big)^{-1} \right\VERT_{2} \bigg) \\
    &\leq \frac{s}{C_{\text{min}}}\bigg( \frac{1}{NM} \left\VERT \widehat{\bF}_{\mS^{c}}^{T}\widehat{\bF}_{\mS}-{\bF}_{\mS^{c}}^{T}{\bF}_{\mS} \right\VERT_{2} \bigg).
\end{align*}
\noindent\textbf{\textit{3. Control of $\bf{T_{3}}$}}: To bound $\left\VERT\bf{T_{3}}\right\VERT_{\infty}$, we re-factorize the second argument of product in $\bf{T_{3}}$: 
\begin{equation*}
    \big(\widehat{\bF}_{\mS}^{T}\widehat{\bF}_{\mS} \big)^{-1}-\big( \bF_{\mS}^{T}\bF_{\mS} \big)^{-1}
    = \big( \bF_{\mS}^{T}\bF_{\mS} \big)^{-1}
    \big[ \big( \bF_{\mS}^{T}\bF_{\mS} \big) - \big(\widehat{\bF}_{\mS}^{T}\widehat{\bF}_{\mS} \big) \big] 
    \big(\widehat{\bF}_{\mS}^{T}\widehat{\bF}_{\mS} \big)^{-1}
\end{equation*}
With the factorization, we bound $\VERT \big(\widehat{\bF}_{\mS}^{T}\widehat{\bF}_{\mS} \big)^{-1}-\big( \bF_{\mS}^{T}\bF_{\mS} \big)^{-1} \VERT_{\infty}$ by using sub-multiplicative property and the fact $\VERT C \VERT_{\infty}\leq \sqrt{N} \VERT C \VERT_{2}$ for any $C \in \mathbb{R}^{M \times N}$ again:
\begin{align}
    \left\VERT \big(\widehat{\bF}_{\mS}^{T}\widehat{\bF}_{\mS} \big)^{-1}-\big( \bF_{\mS}^{T}\bF_{\mS} \big)^{-1} \right\VERT_{\infty}
    &= \left\VERT \big( \bF_{\mS}^{T}\bF_{\mS} \big)^{-1}
    \big[ \big( \bF_{\mS}^{T}\bF_{\mS} \big) - \big(\widehat{\bF}_{\mS}^{T}\widehat{\bF}_{\mS} \big) \big] 
    \big(\widehat{\bF}_{\mS}^{T}\widehat{\bF}_{\mS} \big)^{-1} \right\VERT_{\infty}  \nonumber \\
    &\leq \sqrt{s} \left\VERT \big( \bF_{\mS}^{T}\bF_{\mS} \big)^{-1}
    \big[ \big( \bF_{\mS}^{T}\bF_{\mS} \big) - \big(\widehat{\bF}_{\mS}^{T}\widehat{\bF}_{\mS} \big) \big] 
    \big(\widehat{\bF}_{\mS}^{T}\widehat{\bF}_{\mS} \big)^{-1} \right\VERT_{2} \nonumber  \\
    &\leq \sqrt{s}\left\VERT \big( \bF_{\mS}^{T}\bF_{\mS} \big)^{-1}  \right\VERT_{2}
    \left\VERT \big[ \big( \bF_{\mS}^{T}\bF_{\mS} \big) - \big(\widehat{\bF}_{\mS}^{T}\widehat{\bF}_{\mS} \big) \big]  \right\VERT_{2}
    \left\VERT \big(\widehat{\bF}_{\mS}^{T}\widehat{\bF}_{\mS} \big)^{-1} \right\VERT_{2} \nonumber \\
    &\leq \frac{\sqrt{s}}{NM C_{\text{min}}^{2}}\bigg(\frac{1}{NM}\left\VERT \bF_{\mS}^{T}\bF_{\mS} - \widehat{\bF}_{\mS}^{T}\widehat{\bF}_{\mS} \right\VERT_{2} \bigg). \label{useful}
\end{align}
In the last inequality, we use the result of Lemma~\ref{lem.eign'}.
Now we can bound $\left\VERT\bf{T_{3}}\right\VERT_{\infty}$ as follows:
\begin{align*}
    \left\VERT\bf{T_{3}}\right\VERT_{\infty}&=
    \left\VERT \bigg( \widehat{\bF}_{\mS^{c}}^{T}\widehat{\bF}_{\mS}-{\bF}_{\mS^{c}}^{T}{\bF}_{\mS} \bigg)\bigg( \big(\widehat{\bF}_{\mS}^{T}\widehat{\bF}_{\mS} \big)^{-1}-\big( \bF_{\mS}^{T}\bF_{\mS} \big)^{-1} \bigg) \right\VERT_{\infty} \\
    &\leq \left\VERT \widehat{\bF}_{\mS^{c}}^{T}\widehat{\bF}_{\mS}-{\bF}_{\mS^{c}}^{T}{\bF}_{\mS} \right\VERT_{\infty} 
    \left\VERT \big(\widehat{\bF}_{\mS}^{T}\widehat{\bF}_{\mS} \big)^{-1}-\big( \bF_{\mS}^{T}\bF_{\mS} \big)^{-1}  \right\VERT_{\infty} \\
    &\leq \frac{s}{C_{\text{min}}} \bigg( \frac{1}{NM}\left\VERT \widehat{\bF}_{\mS^{c}}^{T}\widehat{\bF}_{\mS}-{\bF}_{\mS^{c}}^{T}{\bF}_{\mS} \right\VERT_{2} \bigg) 
    \bigg( \frac{1}{NM}  \left\VERT \bF_{\mS}^{T}\bF_{\mS} - \widehat{\bF}_{\mS}^{T}\widehat{\bF}_{\mS} \right\VERT_{2} \bigg),
\end{align*}
where in the last inequality, we use \eqref{useful} and $\VERT C \VERT_{\infty}\leq \sqrt{N} \VERT C \VERT_{2}$ for any $C \in \mathbb{R}^{M \times N}$.
Take $\varepsilon_{M}^{\max^{''}}$ such that, for $\varepsilon_{M}^{\max'}$ and $\varepsilon_{M}^{\max}$ in Lemma~\ref{lemm.Fsc_conv} and Lemma \ref{lem.eign'} respectively:
\begin{equation*}
\varepsilon_{M}^{\text{max}^{''}}>\max\bigg\{ \frac{C_{\min}}{s(1-\mu)}\varepsilon_{M}^{\max}, \frac{C_{\min}}{s}\varepsilon_{M}^{\max^{'}}  \bigg\}\;,
\end{equation*}
for large enough $N$, we have
\begin{align*}
    \mathbb{P}\Bigg[\forall i = 1,2,3: \left\VERT\bf{T_{i}} \right\VERT_\infty >\varepsilon_{M}^{\text{max}^{''}}\Bigg] \leq \mathcal{O}\bigg(N\exp\big(-\frac{1}{6}N^c\big)\bigg)\;.
\end{align*}

\end{proof}
\end{document}